\documentclass[11pt,reqno]{amsart}

\usepackage{amsmath,amssymb,amsthm}
\usepackage{thmtools}
\usepackage{mathtools}
\usepackage{enumerate}

\usepackage[german,english]{babel}
\usepackage[autostyle]{csquotes}

\usepackage{geometry}
\usepackage{tikz-cd}
\usetikzlibrary{babel}

\usepackage{hyperref}
\hypersetup{
	colorlinks=true,
	linkcolor=red,
	citecolor=blue}

\theoremstyle{plain}

\newtheorem{thm}{Theorem}[section]
\newtheorem{prop}[thm]{Proposition}
\newtheorem{lem}[thm]{Lemma}
\newtheorem{cor}[thm]{Corollary}
\newtheorem{conj}[thm]{Conjecture}

\theoremstyle{definition}

\newtheorem{dfn}[thm]{Definition}
\newtheorem{exa}[thm]{Example}
\newtheorem{rem}[thm]{Remark}
\newtheorem{question}[thm]{Question}

\theoremstyle{plain}

\newtheorem{thmA}{Theorem}

\newtheorem{corA}[thmA]{Corollary}

\newcommand{\N}{\mathbb{N}}
\newcommand{\Z}{\mathbb{Z}}
\newcommand{\Q}{\mathbb{Q}}
\newcommand{\R}{\mathbb{R}}

\newcommand{\K}{\mathbb{K}}
\newcommand{\OO}{\mathcal{O}}

\DeclareMathOperator{\Exc}{Exc}
\DeclareMathOperator{\mult}{mult}
\DeclareMathOperator{\Supp}{Supp}
\DeclareMathOperator{\Pic}{Pic}
\DeclareMathOperator{\Div}{Div}

\DeclareMathOperator{\Bs}{Bs} 
\DeclareMathOperator{\sbs}{\mathbf{B}} 
\DeclareMathOperator{\dbs}{\mathbf{B}_{--}} 
\DeclareMathOperator{\abs}{\mathbf{B}_{+}} 
\DeclareMathOperator{\nnef}{NNef} 

\begin{document}
	
	\title
	{Comparison and Uniruledness of Asymptotic Base Loci}
	
	\author[Nikolaos Tsakanikas]{Nikolaos Tsakanikas}
	\address{Institut de Math\'ematiques (CAG), \'Ecole Polytechnique F\'ed\'erale de Lausanne (EPFL), 1015 Lausanne, Switzerland}
	\email{nikolaos.tsakanikas@epfl.ch}
	
	\author[Zhixin Xie]{Zhixin Xie}
	\address{Institut \'Elie Cartan de Lorraine, Universit\'e de Lorraine, 54506 Nancy, France}
	\email{zhixin.xie@univ-lorraine.fr}
	
	\thanks{N.T.\ gratefully acknowledges support by the ERC starting grant \#804334. This project was initiated during N.T.'s research stay from October 2022 to March 2023 at the Max Planck Institute for Mathematics (MPIM) in Bonn. N.T.\ is grateful to the MPIM for its hospitality and financial support. \newline
	}
	
	\begin{abstract}
		We prove that the asymptotic base loci of an NQC klt generalized pair with big canonical class are uniruled. We also show that the non-nef locus and the diminished base locus of the adjoint divisor of an NQC log canonical generalized pair coincide. As applications, we study the uniruledness of the asymptotic base loci associated with pseudo-effective divisors on generalized log Calabi-Yau type varieties.
	\end{abstract}
	
	\maketitle
	
	\begingroup
	\hypersetup{linkcolor=black}
	\setcounter{tocdepth}{1}
	\tableofcontents
	\endgroup

	\section{Introduction}
	\label{section:intro}
	
	Throughout the paper we work over an algebraically closed field of characteristic $0$. Unless otherwise stated, we assume that varieties are normal and projective.
	
	\medskip
	
	The base locus of a linear system is a fundamental object of study in algebraic geometry, since it carries important geometric information about the linear system itself. To analyze the asymptotic behavior of the linear system, several refined asymptotic notions of base loci have been introduced in the literature. Specifically, given a variety $X$ and an $\R$-Cartier $\R$-divisor $D$ on $X$, one defines the \emph{stable base locus} of $D$ as
	\[ \sbs(D) \coloneqq \bigcap \big\{ \Supp E \mid E \text{ is an effective $\R$-divisor on $X$ such that } E \sim_\R D \big\} , \]
	the \emph{augmented base locus} of $D$ as
	\[ \abs(D) \coloneqq \bigcap_A \sbs(D-A) , \]
	and the \emph{diminished base locus} of $D$ as
	\[ \dbs(D) \coloneqq \bigcup_A \sbs(D+A) , \]
	where the intersection, respectively the union, is taken over all ample $\R$-divisors $A$ on $X$. Note that $\abs(D) = \emptyset$ if and only if $D$ is ample, and that $\dbs(D) = \emptyset$ if and only if $D$ is nef. Therefore, these two loci measure the failure of positivity properties of the $\R$-Cartier $\R$-divisor $D$.
	Furthermore, we have the inclusions
	\[ \dbs(D) \subseteq \sbs(D) \subseteq \abs(D) , \]
	which are strict in general: we actually construct in Example \ref{exa:dbs_vs_sbs_vs_abs} a $\Z$-divisor $D$ on a smooth projective surface $S$ such that
	\[ \emptyset \subsetneq \dbs(D) \subsetneq \sbs(D) \subsetneq \abs(D) \subsetneq S . \]
	
	To detect the failure of $D$ to be nef, one also introduces the so-called \emph{non-nef locus} $\nnef(D)$ of $D$, which is defined in terms of asymptotic vanishing orders attached to $D$ (see Definitions \ref{dfn:avo} and \ref{dfn:nnef}). This locus is empty precisely when $D$ is nef, and we have the inclusion $\nnef(D) \subseteq \dbs(D)$. It is conjectured in \cite{BBP13} that these two loci always coincide. Nakayama \cite{Nak04} confirmed this conjecture when $X$ is a smooth projective variety, while Boucksom, Broustet and Pacienza \cite{BBP13} verified their conjecture in the klt setting for the adjoint divisor.
	Moreover, Cacciola and Di Biagio \cite{CDB13} established this conjecture in dimension two, as well as in higher dimensions under the assumption that $X$ admits the structure of a klt pair with rational coefficients.
	In Section \ref{section:comparison_loci} we extend the previous results to the category of generalized pairs; in particular, we show that Boucksom, Broustet and Pacienza's conjecture holds for the adjoint divisor of a log canonical generalized pair (see Theorem \ref{thm:loci_comparison_can}(i)).
	
	Generalized pairs were introduced by Birkar and Zhang \cite{BZ16} and have played a central role in the latest developments in the Minimal Model Program (MMP), see for instance \cite{HanLiu20,Hash22a,LT22b,Xie22,HL23,LX23b,TX23a}.
	Loosely speaking, a \emph{generalized pair} (g-pair for short) is a couple of the form $(X,B+M)$, where $(X,B)$ is a usual pair and $M$ is an auxiliary $\R$-divisor on $X$ with certain positivity properties. When one works with g-pairs with real coefficients, the $\R$-divisor $M$ is usually required to satisfy the  \emph{NQC condition}, which guarantees that the given g-pair behaves similarly to a usual pair, particularly with regards to MMP considerations. The precise definition of an \emph{NQC g-pair} is given in Definition \ref{dfn:NQC_g-pair}.
	
	\medskip
	
	The central objective of this paper is to investigate the uniruledness of the asymptotic base loci discussed above in the context of generalized pairs. This problem was studied previously by Takayama \cite{Tak08} and by Boucksom, Broustet and Pacienza \cite{BBP13} in the setting of varieties or pairs, motivated by the fact that the geometry of varieties is intimately related to the existence of rational curves on them. For instance, klt pairs with rational coefficients and nef and big anti-canonical class are rationally connected \cite{Zhang06}; and the fibers of any birational morphism to a dlt pair with rational coefficients are uniruled \cite{HM07a}.
	We emphasize that the analogues of the previous two statements in the strictly log canonical setting fail in general (see Example \ref{exa:blowup_vertex_of_cone}).
	
	Our main result concerning the uniruledness of the asymptotic base loci is the following generalization of both \cite[Theorem 1.1]{Tak08} and \cite[Theorem A]{BBP13} to the category of generalized pairs.
	
	\begin{thmA}\label{thm:loci_uniruledness_can}
		Let $(X,B+M)$ be an NQC klt g-pair such that $K_X+B+M$ is pseudo-effective. The following statements hold:
		\begin{enumerate}[\normalfont (i)]
			\item Every irreducible component of 
			\[ \nnef(K_X+B+M) = \dbs(K_X+B+M) \]
			is uniruled.
			
			\item If $K_X+B+M$ is big, then the irreducible components of both
			\[ \nnef(K_X+B+M) = \dbs(K_X+B+M) = \sbs(K_X+B+M) \]
			and $\abs(K_X+B+M)$ are uniruled.
			
			\item Assume that $|K_X+B+M|_\R \neq \emptyset$. If $B$ or $M$ is big, then the irreducible components of
			\[ \nnef(K_X+B+M) = \dbs(K_X+B+M) = \sbs(K_X+B+M) \]
			are uniruled.
		\end{enumerate}
	\end{thmA}
	
	The main ingredients of the proof of Theorem \ref{thm:loci_uniruledness_can}, which is given in Section \ref{section:uniruledness_loci}, are the uniruledness of the irreducible components of the exceptional locus of an MMP step 
	(see Lemma \ref{lem:reduction_to_Q-pairs}(ii))
	and of a birational klt-trivial fibration
	(see Lemma \ref{lem:exc_locus_uniruled}).
	As a by-product of the latter, we show that the irreducible components of the augmented base locus of an NQC klt g-pair with nef and big canonical or anti-canonical class are uniruled (see Corollaries \ref{cor:Bir17_Q} and \ref{cor:abs_uniruledness_weak_Fano}). Note that Corollary \ref{cor:Bir17_Q} provides an affirmative answer to \cite[Question 6.4]{Bir17} over an algebraically closed field of characteristic $0$.
	
	We also present some examples in Section \ref{section:uniruledness_loci} in order to demonstrate that our results in Theorem \ref{thm:loci_uniruledness_can} are optimal. In particular, it turns out that the irreducible components of the stable or the augmented base locus of an NQC lc g-pair need not be uniruled in general. Nevertheless, we show in Section \ref{section:uniruledness_dbs_lc} that the irreducible components of the diminished base locus of an NQC lc g-pair are usually uniruled (see Theorem \ref{thm:uniruledness_dbs_can_lc}).
	
	\medskip
	
	An application of Theorem \ref{thm:loci_uniruledness_can} is the following result, which addresses the uniruledness of the asymptotic base loci associated with an arbitrary pseudo-effective $\R$-divisor on a generalized log Calabi-Yau type variety.
	
	\begin{corA}\label{cor:BBP13_CorA_g}
		Let $(X,B+M)$ be an NQC klt g-pair such that $K_X+B+M \equiv 0$ and let $D$ be an $\R$-Cartier $\R$-divisor on $X$. The following statements hold:
		\begin{enumerate}[\normalfont (i)]
			\item If $D$ is pseudo-effective, then the irreducible components of 
			\[ \nnef(D) = \dbs(D) \]
			are uniruled.
			
			\item If $D$ is big and if $K_X+B+M\sim_\R 0$, then the irreducible components of
			\[ \nnef(D) = \dbs(D) = \sbs(D) \quad \text{ and } \quad \abs(D) \]
			are uniruled.
			
			\item Assume that $|D|_\R \neq \emptyset$. If $B$ or $M$ is big and if $K_X+B+M \sim_\R 0$, then the irreducible components of
			\[ \nnef(D) = \dbs(D) = \sbs(D) \]
			are uniruled.
		\end{enumerate}
	\end{corA}
	
	Corollary \ref{cor:BBP13_CorA_g} improves both \cite[Theorem 1.2]{Tak08} and \cite[Corollary A]{BBP13}. In particular, it implies that if $(X,B)$ is a klt pair such that $-(K_X+B)$ is nef and if $D$ is a pseudo-effective (resp.\ big) $\R$-divisor on $X$, then the irreducible components of $\dbs(D)$ (resp.\ both $\dbs(D) = \sbs(D)$ and $\abs(D)$) are uniruled.
	
	Finally, we highlight that the hypothesis that the given $\R$-divisor is big can be replaced in both Theorem \ref{thm:loci_uniruledness_can}(ii) and Corollary \ref{cor:BBP13_CorA_g}(ii) by the weaker assumption that it is \emph{abundant}, i.e., its invariant Iitaka dimension and numerical dimension coincide. More precisely, we obtain the following refined statements, which are new even in the setting of usual pairs.
	
	\begin{corA}\label{cor:loci_uniruledness_abundant}
		Let $(X,B+M)$ be an NQC klt g-pair. The following statements hold:
		\begin{enumerate}[\normalfont (i)]
			\item If $K_X+B+M$ is pseudo-effective and abundant, then the irreducible components of 
			\[ \nnef(K_X+B+M) = \dbs(K_X+B+M) = \sbs(K_X+B+M) \]
			are uniruled.
			
			\item If $K_X+B+M\sim_\R 0$ and if $D$ is a pseudo-effective and abundant $\R$-divisor on $X$, then the irreducible components of 
			\[ \nnef(D) = \dbs(D) = \sbs(D) \]
			are uniruled.
		\end{enumerate}
	\end{corA}
	
	\emph{Acknowledgements}: We would like to thank J.\ Baudin, F.\ Bernasconi, P.\ Chaudhuri, S.\ Filipazzi, E.\ Floris, L.\ Heuberger, G.\ Martin, N.\ M\"uller, G.\ Pacienza, Z.\ Patakfalvi, J.E.\ Rodr\'iguez Camargo, J.\ Schneider, L.\ Xie, C.\ Zhou and S.\ Zikas for many useful discussions related to this work. We would also like to thank O.\ Fujino for answering our questions and for informing us about \cite[Theorem 1.12]{Fuj21b}, and V.\ Lazi\'c for many valuable comments and suggestions. Finally, we thank the referee for their suggestions to improve the exposition of the paper and for pointing out an easy proof of Theorem \ref{thm:uniruledness_dbs_can_lc}.

	\section{Preliminaries}
	
	We use the terminology of \cite{KM98,Fuj17book} for standard notions in birational geometry (e.g., pairs and their singularities). We refer to \cite{BZ16,HL23,TX23a} and the relevant references therein for the fundamentals of generalized pairs; see also Subsection \ref{subsection:g-pairs}.

	\subsection{Divisors and maps}
	
	Let $X$ be a normal projective variety. Denote by 
	\[ \Div_{\K}(X) \coloneqq \Div(X) \otimes_{\Z} \K \]
	the $\K$-vector space of $\K$-Cartier $\K$-divisors on $X$, where $ \K \in \{ \Q, \R \} $.
	Fix $D\in \Div_{\R}(X)$ and set
	\begin{itemize}
		\item $|D|_\Q \coloneqq \big\{ \Delta \mid \Delta \text{ is an effective $\R$-divisor on $X$ such that } \Delta \sim_\Q D \big\} $,
		
		\item $|D|_\R \coloneqq \big\{ \Delta \mid \Delta \text{ is an effective $\R$-divisor on $X$ such that } \Delta \sim_\R D \big\} $, and 
		
		\item $|D|_\equiv \coloneqq \big\{ \Delta \mid \Delta \text{ is an effective $\R$-divisor on $X$ such that } \Delta \equiv D \big\} $.
	\end{itemize}
	Recall also that $D$ is said to be \emph{NQC}, which stands for \emph{nef $\Q$-Cartier combinations} \cite{HanLi22}, if it is a non-negative linear combination of nef $\Q$-Cartier divisors on $X$.
	
	We denote by $\kappa_\iota(X,D)$ the \emph{invariant Iitaka dimension} of $D$ and by $\kappa_\sigma(X,D)$ the \emph{numerical dimension} of $D$, see \cite[Chapter V]{Nak04} and \cite[Chapter 2]{Fuj17book}. We say that $D$ is \emph{abundant} if the equality $\kappa_\iota(X,D) = \kappa_\sigma(X,D)$ holds. We also recall that $\kappa_\iota(X,D) = \kappa(X,D)$ when $D\in\Div_{\Q}(X)$; see \cite[Proposition 2.5.9]{Fuj17book}.
	
	\medskip
	
	A \emph{fibration} is a projective surjective morphism with connected fibers, and a \emph{birational contraction} is a birational map whose inverse does not contract any divisors. We also recall below the notion of a $D$-non-positive map.
	
	\begin{dfn}
		Let $\varphi \colon X\dashrightarrow Y$ be a birational contraction between normal varieties, let $D\in\Div_{\R}(X)$ and assume that $\varphi_*D$ is $\R$-Cartier. We say that $\varphi$ is \emph{$D$-non-positive} if there exists a resolution of indeterminacies $(p,q) \colon W \to X \times Y$ of $\varphi$ such that $ W $ is smooth and 
		\[ p^*D \sim_\R q^* (\varphi_*D) + E , \]
		where $E$ is an effective $q$-exceptional $\R$-divisor on $W$. 
	\end{dfn}

	\subsection{Asymptotic base loci}
	\label{subsection:asymptotic_base_loci}
	
	First, we collect below the basic properties of the stable, the diminished and the augmented base loci associated with an $\R$-Cartier $\R$-divisor on a normal projective variety. For more information we refer to \cite{Nak04,ELMNP06,BBP13,CDB13}.
	
	\begin{rem}~
		\label{rem:sbs_properties}
		\begin{enumerate}[(1)]
			\item If $D\in\Div_{\Q}(X)$, then it follows from \cite[Lemma 3.5.3]{BCHM10} that the definition of the stable base locus of $D$ given in Section \ref{section:intro} coincides with the usual definition given in 
			\cite[Section 1, pp.\ 1705-1706]{ELMNP06}. In particular, by \cite[Proposition 2.1.21]{Laz04} there exists a positive integer $m$ such that 
			\[ \sbs(D) = \Bs |kmD| \ \text{ for all } k \gg 0 . \]
			
			\item For any $ \varepsilon \in \R_{>0} $ it holds that
			\[ \sbs(\varepsilon D) = \sbs(D) . \]
			
			\item If $D_1,D_2\in\Div_{\R}(X)$, then
			\[ \sbs(D_1+D_2) \subseteq \sbs(D_1)  \cup\sbs(D_2) . \]
			
			\item If $D\in\Div_{\Q}(X)$, then $D$ is semi-ample if and only if $\sbs(D) = \emptyset$. Therefore, for any semi-ample $\R$-divisor $S$ on $X$ we have $\sbs(S) = \emptyset$.
			
			\item If $f \colon Y \to X$ is a birational morphism from a normal projective variety $Y$, then 
			\[ \sbs(f^* D) = f^{-1} \big( \sbs(D) \big) ; \]
			see the proof of \cite[Lemma 2.3]{LMT23}.
		\end{enumerate}
	\end{rem}
	
	\begin{rem}~
		\label{rem:abs+dbs_properties}
		\begin{enumerate}[(1)]
			\item The loci $\abs(D)$ and $\dbs(D)$ depend only on the numerical equivalence class of $D$; see \cite[Proposition 1.4 and Proposition 1.15(ii)]{ELMNP06}, respectively. However, this does not hold in general for $\sbs(D)$; see \cite[Example 1.1]{ELMNP06}.
			
			\item For any $ \varepsilon \in \R_{>0} $, by \cite[Example 1.8 and Proposition 1.15(i)]{ELMNP06} we have 
			\[ \abs( \varepsilon D ) = \abs(D) \quad \text{ and } \quad \dbs( \varepsilon D ) = \dbs(D). \]
			
			\item By \cite[Examples 1.7 and 1.18]{ELMNP06} the following hold:
			\begin{itemize}
				\item $\abs(D) = \emptyset$ if and only if $D$ is ample;
				
				\item $\abs(D) = X$ if and only if $D$ is not big;
				
				\item $\dbs(D) = \emptyset$ if and only if $D$ is nef;
				
				\item $\dbs(D) = X$ if and only if $D$ is not pseudo-effective.
			\end{itemize}
			
			\item It follows from \cite[Proposition 1.1]{ELMNP09} and \cite[Lemmata 2.12 and 2.13]{CDB13} that the loci $\dbs(D)$, $\sbs(D)$ and $\abs(D)$ do not contain any isolated points.
			
			\item Let $f \colon Y \to X$ be a birational morphism from a normal projective variety $Y$. If $D \in \Div_\R(X)$ is big and if $E \in \Div_\R(Y)$ is effective and $f$-exceptional, then by \cite[Proposition 2.3]{BBP13} we have
			\[ \abs( f^* D + E ) = f^{-1} \big( \abs(D) \big) \cup \Exc(f) . \]
		\end{enumerate}
	\end{rem}
	
	The following two remarks are useful for the computation of the diminished and the augmented base locus of a given divisor.
	
	\begin{rem}
		\label{rem:dbs_negative_intersection}
		Let $X$ be a normal projective variety and let $D \in \Div_{\R}(X)$. Then $\dbs(D)$ contains all curves on $X$ which intersect $D$ negatively. Indeed, this is trivial if $D$ is not pseudo-effective, so we assume now that $D$ is pseudo-effective. If $C$ is a curve on $X$ such that $D \cdot C < 0$, then we can find an ample $\R$-divisor $A$ on $X$ such that $(D+A) \cdot C < 0$, so for every $G \in |D+A|_\R$ we have $G \cdot C < 0$, which implies that $C \subseteq \Supp G$. Hence, $C \subseteq \sbs(D+A) \subseteq \dbs(D)$, which yields the above assertion.
		On the other hand, \cite[Example 1.5]{DO23} indicates that in general $\dbs(D)$ may also contain curves $\gamma$ on $X$ such that $D \cdot \gamma > 0$.
		
		Assume now that the given $\R$-divisor $D$ is big and consider a curve $C$ on $X$. If $C$ is not contained in $\abs(D)$, then $D \cdot C > 0$. However, if $C$ is contained in $\abs(D)$, then it does not necessarily hold that $D \cdot C \leq 0$ in general. We refer to \cite[Remark 2.8]{DO23} for the details.
	\end{rem}
	
	\begin{rem}
		\label{rem:abs_alt_description}
		According to \cite[Example 1.10]{ELMNP06}, if $X$ is a smooth projective variety and if $D$ is a nef and big $\R$-divisor on $X$, then $\abs(D)$ is given as the union of all positive-dimensional subvarieties $V$ of $X$ such that $D^{\dim V} \cdot V = 0$, namely, the union of all positive-dimensional subvarieties of $X$ on which $D$ is not big. We refer to \cite{CDB13,Bir17} for the generalization of the previous result to the singular setting.
	\end{rem}
	
	In the remainder of this subsection we discuss the non-nef locus of an $\R$-Cartier $\R$-divisor on a normal projective variety.
	
	\begin{dfn}
		\label{dfn:avo}
		Let $X$ be a normal projective variety, let $D\in\Div_{\R}(X)$ and let $v$ be a divisorial valuation over $X$. 
		\begin{enumerate}[(a)]
			\item If $D$ is big, set
			\[ v \big( \| D \| \big) \coloneqq \inf \big\{ v(\Delta) \mid \Delta \in |D|_\equiv \big\} . \]
			
			\item If $D$ is pseudo-effective, set
			\[ v \big( \| D \| \big) \coloneqq \lim_{\varepsilon \to 0^+} v \big( \| D + \varepsilon A \| \big) , \]
			where $A$ is an ample $\R$-divisor on $X$. Note that this definition does not depend on the choice of $A$; see the proof of \cite[Lemma III.1.5(2)]{Nak04}.
		\end{enumerate}
		The quantity $ v \big( \| D \| \big) $ is called the \emph{asymptotic vanishing order} of $D$ along $v$.
	\end{dfn}
	
	\begin{rem}
		\label{rem:avo_properties}
		We use the same notation as in Definition \ref{dfn:avo} and we also assume that $D$ is pseudo-effective.
		\begin{enumerate}[(1)]
			\item If $D$ is big, then
			\[ v \big( \| D \| \big) = \inf \big\{ v(\Delta) \mid \Delta \in |D|_\R \big\} ; \]
			see the proof of \cite[Lemma III.1.4(3)]{Nak04}.
			
			\item By \cite[Lemma 2.4]{BBP13}, $ v \big( \| D \| \big) $ is a birational invariant, i.e., for any birational morphism $f \colon Y \to X$ from a normal projective variety $Y$ we have
			\[ v \big( \| D \| \big) = v \big( \| f^*D \| \big) . \]
			
			\item If $\lambda >0$, then 
			\[ v \big( \| \lambda D \| \big) = \lambda v \big( \| D \| \big) . \]
			If $D'$ is another pseudo-effective $\R$-divisor on $X$, then
			\[ v \big( \| D + D' \| \big) \leq v \big( \| D \| \big) + v \big( \| D' \| \big) . \]
		\end{enumerate}
	\end{rem}
	
	\begin{dfn}
		\label{dfn:nnef}
		Let $X$ be a normal projective variety and let $D\in\Div_{\R}(X)$. If $D$ is pseudo-effective, then the \emph{non-nef locus} of $D$ is defined as
		\[ \nnef(D) \coloneqq \bigcup \big\{ c_X(v) \mid v \big( \| D \| \big) > 0 \big\} , \]
		where $c_X(v)$ denotes the center on $X$ of a given divisorial valuation $v$ over $X$. Otherwise, by convention we set $\nnef(D) \coloneqq X$.
	\end{dfn}
	
	\begin{rem}
		\label{rem:nnef_properties}
		We use the same notation as in Definition \ref{dfn:nnef}.
		\begin{enumerate}[(1)]
			\item The locus $\nnef(D)$ depends only on the numerical equivalence class of $D$.
			
			\item For any $ \varepsilon \in \R_{>0} $ it holds that \[ \nnef( \varepsilon D) = \nnef(D) . \]
			
			\item We have $\nnef(D) = \emptyset$ if and only if $D$ is nef; see \cite[Remark III.2.8]{Nak04}.
			
			\item By \cite[Lemma 2.6]{BBP13}, it holds that
			\[ \nnef(D) \subseteq \dbs(D) . \]
			
			\item For any resolution $\pi \colon Y \to X $ of $X$, we have
			\[ \nnef(D) = \pi \big( \nnef(\pi^* D)\big) \]
			by Remark \ref{rem:avo_properties}(2). Since \cite[Lemma V.1.9(1)]{Nak04} implies $ \nnef( \pi^*D ) = \dbs( \pi^*D ) $, by \cite[Proposition 1.19]{ELMNP06} we infer that $\nnef(D)$ is an at most countable union of Zariski-closed subsets of $X$.
		\end{enumerate}
	\end{rem}

	\subsection{Nakayama--Zariski decomposition}
	\label{subsection:NZD}
	
	Given a smooth projective variety $X$ and a pseudo-effective $\R$-divisor $D$ on $X$, Nakayama \cite{Nak04} defined a decomposition 
	$ D = P_\sigma (D) + N_\sigma (D) $, known as the \emph{Nakayama--Zariski decomposition} of $D$. This decomposition has been extended to the singular setting; see for instance \cite[Section 4]{BH14b} or \cite[Section 3]{LX23a}. We recall here its definition for the convenience of the reader.
	
	\begin{dfn}
		Let $X$ be a normal projective variety and let $D$ be a pseudo-effective $\R$-divisor on $X$. Given a prime divisor $\Gamma$ on $X$, denote by $\sigma_\Gamma (D)$ the asymptotic vanishing order of $D$ along $\Gamma$. Set
		\[ N_\sigma (D) \coloneqq \sum_\Gamma \sigma_\Gamma(D) \cdot \Gamma \quad \text{ and } \quad P_\sigma(D) \coloneqq D-N_\sigma(D) , \] 
		where the above formal sum runs through all prime divisors $\Gamma$ on $X$. Both $N_\sigma(D)$ and $P_\sigma(D)$ are $\R$-divisors on $X$, $N_\sigma(D)$ is effective and $P_\sigma(D)$ is movable. The decomposition 
		$$ D = P_\sigma (D) + N_\sigma (D) $$
		is called the \emph{Nakayama--Zariski decomposition} of $D$.
	\end{dfn}
	
	We gather below some basic properties of the Nakayama--Zariski decomposition of a pseudo-effective $\R$-divisor that will be often used in the sequel. For further properties of this decomposition we refer to \cite[Chapter III]{Nak04}.
	
	\begin{rem}
		\label{rem:NZD_properties}
		Let $X$ be a normal projective variety and let $D$ be a pseudo-effective $\R$-divisor on $X$.
		\begin{enumerate}[(1)]
			\item The irreducible components of $N_\sigma(D)$ coincide with the divisorial components of $\dbs(D)$; see \cite[Lemma 3.7(4)]{LX23a}.
			
			\item If $D$ is movable, then $N_\sigma(D) = 0$; see \cite[Lemma 4.1(6)]{BH14b}.
			
			\item If $f \colon Y \to X$ is a birational morphism from a normal projective variety $Y$ and if $E$ is an effective $f$-exceptional $\R$-Cartier $\R$-divisor on $Y$, then
			\[ N_\sigma \big( f^*D + E \big) = N_\sigma \big( f^*D \big) + E ; \]
			see \cite[Lemma 4.1(2)]{BH14b} or \cite[Lemma 2.4]{LP20a}.
			In particular, we have $ N_\sigma(E) = E $. Therefore, \[\dbs(E) = \sbs(E) = \Supp(E) . \]
		\end{enumerate}
	\end{rem}

	\subsection{Generalized pairs and minimal models}
	\label{subsection:g-pairs}
	
	For the convenience of the reader we first recall the definition of generalized pairs, as well as the definition of two fundamental classes of singularities of generalized pairs (i.e., Kawamata log terminal and log canonical), which will be used throughout the paper.
	
	\begin{dfn}
		\label{dfn:NQC_g-pair}
		A \emph{generalized pair}, abbreviated as \emph{g-pair}, consists of 
		\begin{itemize}
			\item a normal projective variety $ X $,
			
			\item an effective $ \R $-divisor $ B $ on $X$,
			
			\item a projective birational morphism $ f \colon W \to X $ from a normal variety $ W $ and an $\R$-Cartier $\R$-divisor $ M_W $ on $ W $ which is nef,
		\end{itemize}
		such that the $\R$-divisor $ K_X + B + M $ is $ \R $-Cartier, where $ M \coloneqq f_* M_W $.
		We denote a g-pair simply by $(X,B+M)$, but remember the whole g-pair structure, and we call it \emph{NQC} if $M_W$ is an NQC $\R$-divisor on $ W $, i.e., $M_W = \sum_{j=1}^\ell \mu_j M_j'$, where $\mu_j \in \R_{\geq 0}$ and $M_j'$ are nef $\Q$-Cartier divisors on $W$.
	\end{dfn}
	
	Recall that a g-pair $(X,B+M)$ with data $\operatorname{Id}_X \colon X \to X$ and $M_X = M$ is called a \emph{polarized pair}, see \cite{BH14b}, while by taking $M_W = M = 0$ one recovers the notion of a usual \emph{pair}, see \cite{KM98,Fuj17book}. When all $\R$-divisors involved are actually $\Q$-divisors, we use the term \emph{$\Q$-g-pair} to refer to such a generalized pair; in this case, the NQC condition is automatically satisfied.
	
	\begin{dfn}
		Let $ (X,B+M) $ be a g-pair with data $ f \colon W \to X $ and $ M_W $. Let $ E $ be a prime divisor on some birational model of $ X $. After replacing $W$, we may assume that $ E $ is a prime divisor on $ W $. If we write 
		\[ K_W + B_W + M_W \sim_\R f^* ( K_X + B + M ) \]
		for some $ \R $-divisor $ B_W $ on $ W $, then the \emph{discrepancy of $ E $ with respect to $ (X,B+M) $} is defined as
		\[ a(E, X, B+M) \coloneqq - \mult_E B_W . \]
		
		We say that the g-pair $ (X,B+M) $ is:
		\begin{enumerate}[(a)]
			\item \emph{klt} if $  a(E,X,B+M) > -1 $ for any divisorial valuation $ E $ over $ X $, and
			
			\item \emph{lc} if $  a(E,X,B+M) \geq -1 $ for any divisorial valuation $ E $ over $ X $.
		\end{enumerate}
	\end{dfn}
	
	In particular, if $(X,B+M)$ is a polarized pair, that is, $M$ is a nef $\R$-divisor on $X$, then it follows readily from the definition that $(X,B+M)$ is klt (resp.\ lc) if and only if the underlying pair $(X,B)$ is klt (resp.\ lc). This fact will be used in the paper without explicit mention.
	
	Given a (not necessarily $\Q$-factorial) NQC lc g-pair $(X,B+M)$, by the Cone theorem \cite[Theorem 1.1]{HL23}, the Contraction theorem \cite[Theorem 1.5]{Xie22}, \cite[Theorem 1.7]{CLX23}, and the existence of flips \cite[Theorem 1.2]{LX23b}, we may run a $(K_X+B+M)$-MMP, although its termination is not known in general.
	
	\medskip
	
	The next result allows us to reduce certain statements about NQC klt g-pairs with real coefficients to analogous results about klt pairs with rational coefficients.
	
	\begin{lem}
		\label{lem:reduction_to_Q-pairs}
		Let $(X,B+M)$ be an NQC klt g-pair.
		\begin{enumerate}[\normalfont (i)]
			\item There exists an effective $\Q$-divisor $B'$ on $X$ such that $(X,B')$ is a klt $\Q$-pair.
			
			\item Assume that $K_X+B+M$ is not nef and consider the contraction $g \colon X \to Y$ of a $(K_X+B+M)$-negative extremal ray $R$. Then there exists an effective $\Q$-divisor $\Delta$ on $X$ such that $(X,\Delta)$ is a klt $\Q$-pair and $(K_X+\Delta)\cdot R < 0$. In particular, every irreducible component of $\Exc(g)$ is uniruled, i.e., it is covered by rational curves.
		\end{enumerate}
	\end{lem}
	
	\begin{proof}~
		
		\medskip
		
		\noindent (i) The assertion follows from \cite[Lemma 3.4]{HanLi22} and \cite[Theorem 1.4]{Chen23}.
		
		\medskip
		
		\noindent (ii) Since $ (K_X+B+M) \cdot R < 0 $, we may find an effective ample $\R$-divisor $A$ on $X$ such that $ (K_X+B+M+A) \cdot R < 0 $ and $\big(X, (B+A)+M\big)$ is an NQC klt g-pair. Thus, by \cite[Lemma 3.4]{HanLi22}, there exists an effective $\R$-divisor $\Gamma$ on $X$ such that $(X,\Gamma)$ is a klt pair and 
		$ K_X+\Gamma \sim_\R K_X+B+A+M $, so we have
		\begin{equation}\label{eq:1_reduct}
			(K_X + \Gamma) \cdot R < 0 .
		\end{equation}
		
		Next, by \cite[Theorem 1.4]{Chen23} we find positive real numbers $r_1,\dots,r_\ell$ and effective $\Q$-divisors $\Gamma_1,\dots, \Gamma_\ell$ on $X$ such that
		$ \sum_{j=1}^\ell r_j=1 $, $ \Gamma = \sum_{j=1}^\ell r_j \Gamma_j $,
		$(X,\Gamma_j)$ is a klt $\Q$-pair for every $ j \in \{ 1, \dots, \ell \} $, and we have
		\begin{equation}\label{eq:2_reduct}
			K_X + \Gamma = \sum_{j=1}^\ell r_j \big( K_X + \Gamma_j \big) . 
		\end{equation}
		Since $\rho(X/Y) = 1$, for each $ j \in \{ 1, \dots, \ell \} $ there exists $\alpha_j \in \R$ such that
		\begin{equation}\label{eq:3_reduct}
			K_X+\Gamma \equiv_Y \alpha_j (K_X+\Gamma_j) .
		\end{equation}
		Since $-(K_X+\Gamma)$ is ample over $Y$ by \eqref{eq:1_reduct}, it holds that $\alpha_j \neq 0 $ for every $ j \in \{ 1, \dots, \ell \} $. Moreover, at least one of the $\alpha_j$, say $\alpha_1$, must be positive: otherwise by \eqref{eq:3_reduct} each $\Q$-divisor $K_X+\Gamma_j$ is ample over $Y$, and hence the $\R$-divisor $K_X+\Gamma$ is also ample over $Y$ by \eqref{eq:2_reduct}, which contradicts \eqref{eq:1_reduct}. Therefore, the $\Q$-divisor ${-}(K_X+\Gamma_1)$ is ample over $Y$. In conclusion, the $\Q$-divisor $\Delta \coloneqq \Gamma_1$ has the desired properties.
		
		The last part of the statement follows immediately from \eqref{eq:1_reduct} and \cite[Theorem 1]{Kaw91}; see also \cite[Corollary 1.3(2)]{HM07a} and \cite[Theorem 1.12]{Fuj21b}.
	\end{proof}
	
	For the sake of completeness and for future reference, we prove a variant of Lemma \ref{lem:reduction_to_Q-pairs}(ii) in the lc setting, although this result will not be needed in the paper.
	
	\begin{lem}
		\label{lem:exc_locus_uniruled_generalizations}
		The following statements hold:
		\begin{enumerate}[\normalfont (i)]
			\item Let $(X,B)$ be an lc pair. Assume that $K_X+B$ is not nef and consider the contraction $g \colon X \to Y$ of a $(K_X+B)$-negative extremal ray. Then every irreducible component of $\Exc(g)$ is uniruled.
			
			\item Let $(X,B+M)$ be an NQC lc g-pair such that $(X,0)$ is klt. Assume that $K_X+B+M$ is not nef and consider the contraction $g \colon X \to Y$ of a $(K_X+B+M)$-negative extremal ray. Then every irreducible component of $\Exc(g)$ is uniruled.
		\end{enumerate}
	\end{lem}
	
	\begin{proof}~
		
		\medskip
		
		\noindent (i) The assertion follows from \cite[Corollary 1.3(2)]{HM07a} using the same argument as in the second paragraph of the proof of Lemma \ref{lem:exc_locus_uniruled}; see also \cite[Theorem 1.12]{Fuj21b}.
		
		\medskip
		
		\noindent (ii) We may argue as in the first paragraph of the proof of Lemma \ref{lem:reduction_to_Q-pairs}(ii) and then use, for example, (i) to conclude.
	\end{proof}
	
	According to Lemma \ref{lem:reduction_to_Q-pairs}(i) and its proof, a normal projective variety $X$ admits the structure of a klt g-pair if and only if it admits the structure of a klt pair, even with rational coefficients. On the other hand, the projectivized version of \cite[Example 2.1]{LX23b} shows that the analogous statement fails in the lc setting; in other words, there exist normal non-$\Q$-factorial projective varieties which admit the structure of an lc g-pair, but no structure of an lc pair. It is thus natural to ask whether Lemma \ref{lem:exc_locus_uniruled_generalizations}(ii) still holds if one drops the assumption that $(X,0)$ is klt.
	
	\medskip
	
	For the definition of (good) minimal models in the usual sense and in the sense of Birkar-Shokurov of g-pairs we refer to \cite[Subsection 2.2]{TX23a} and we also emphasize that these notions are equivalent for NQC lc g-pairs by \cite[Theorem 2.7 and Lemma 2.10]{TX23a}.
	We briefly discuss below the problem of the existence of \emph{good} minimal models for g-pairs.
	
	Example \ref{exa:dbs_vs_sbs_vs_abs} and \cite[Example 1.5]{LX23a} show that abundance fails in general for NQC lc g-pairs, even if their canonical class is nef and big or nef and log abundant. Thus, such g-pairs do not have \emph{good} minimal models in general, even though they are expected to have minimal models according to \cite[Theorem B]{TX23a}.
	
	However, the situation is considerably better for klt g-pairs. Although Example \ref{exa:uniruledness_fails_I} indicates that abundance also fails in general for NQC klt g-pairs, we prove below that NQC klt g-pairs with pseudo-effective and abundant canonical class have \emph{good} minimal models; see also \cite[\S 2.2.2]{LX23b} and \cite[Theorem 2]{Chaud23b}. We would like to thank P.\ Chaudhuri and L.\ Xie for useful relevant discussions.
	
	\begin{thm}
		\label{thm:EGMM_klt_abundant}
		If $(X,B+M)$ is an NQC klt g-pair such that $K_X+B+M$ is pseudo-effective and abundant, then $(X,B+M)$ has a good minimal model.
	\end{thm}
	
	\begin{proof}
		Since $K_X+B+M$ is pseudo-effective and abundant, we have
		\begin{equation}\label{eq:1_EGMM_klt_abundant}
			\kappa_\iota(X,K_X+B+M) = \kappa_\sigma(X,K_X+B+M) \geq 0 .
		\end{equation}
		In particular, there exists an effective $\R$-divisor $D$ on $X$ such that $K_X+B+M \sim_\R D$. Consider the Iitaka fibration $\varphi \colon X \dashrightarrow V$ associated with $D$ and note that 
		\begin{align}\label{eq:2_EGMM_klt_abundant}
			\dim V &= \kappa_\sigma(X,D) = \kappa(X,D) \\
			&= \kappa_\iota(X,K_X+B+M) = \kappa_\sigma(X,K_X+B+M) , \notag
		\end{align}
		taking also \cite[Lemma 2.3(5)]{LX23a} into account. Pick a sufficiently high log resolution $f \colon Y \to X$ of $(X,B+M)$, which resolves the indeterminacies of $\varphi$, and denote by $\pi $ the induced morphism $ Y \to V $. Since $(X,B+M)$ is klt, we may write
		\begin{equation}\label{eq:3_EGMM_klt_abundant}
			K_Y + B_Y + M_Y \sim_\R f^*(K_X+B+M) + E ,
		\end{equation}
		where $B_Y$ is an effective $\R$-divisor on $Y$ with coefficients $<1$ such that $f_* B_Y = B$, and $E$ is an effective $f$-exceptional $\R$-divisor on $Y$ which has no common components with $B_Y$. In particular, $(Y,B_Y+M_Y)$ is a log smooth klt g-pair, and it follows from \eqref{eq:1_EGMM_klt_abundant}, \eqref{eq:2_EGMM_klt_abundant}, \eqref{eq:3_EGMM_klt_abundant} and \cite[Lemma 2.3(3)]{LX23a} that
		\[ \kappa_\iota(Y,K_Y+B_Y+M_Y) = \kappa_\sigma(Y,K_Y+B_Y+M_Y) = \dim V . \]
		
		Now, since $\varphi$ is the Iitaka fibration associated with $D$, there exist an ample $\R$-divisor $A$ on $V$ and an effective $\R$-divisor $F$ on $Y$ such that $f^*D \sim_\R \pi^*A + F$, so \eqref{eq:3_EGMM_klt_abundant} yields $K_Y+B_Y+M_Y \sim_{\R,\pi} E + F \geq 0 $, and hence
		$ \kappa_\iota(Y/V, K_Y+B_Y+M_Y) \geq 0 . $
		On the other hand, by \cite[Lemma 2.14]{Hash22a} we obtain
		$ \kappa_\sigma(Y/V, K_Y+B_Y+M_Y) = 0 . $
		Therefore, taking \cite[Definitions 2.11 and 2.12]{Hash22a} and \cite[Proposition V.2.7(2)]{Nak04} into account, by replacing $\pi \colon Y \to V$ with its Stein factorization, we may assume that $\pi$ is a fibration and we also infer that
		\[ \kappa_\iota(Y/V, K_Y+B_Y+M_Y) = \kappa_\sigma(Y/V, K_Y+B_Y+M_Y) = 0 . \]
		
		In conclusion, we may apply \cite[Lemma 3.10]{Hash22a} to deduce that $(Y,B_Y+M_Y)$ has a good minimal model in the sense of Birkar-Shokurov, and thus a good minimal model by \cite[Theorem A and Lemma 2.10]{TX23a}. It follows now from \eqref{eq:3_EGMM_klt_abundant}, (the $\R$-divisor version of) \cite[Lemma 2.4]{LP20a} and \cite[Theorem 5.18(ii)]{TX23a} that $(X,B+M)$ has a good minimal model, as asserted.
	\end{proof}

	\subsection{Uniruled varieties}
	
	For the definition and the basic properties of uniruled varieties we refer to \cite{Kol96}. We frequently use below the fact that uniruledness is a birationally invariant property.
	
	The following two lemmata furnish sufficient conditions for the uniruledness of the underlying variety of a g-pair. Lemma \ref{lem:uniruled_nonpsef_can} generalizes an argument extracted from the first paragraph of the proof of \cite[Theorem A]{BBP13}, while Lemma \ref{lem:uniruled_psef_antican} extends \cite[Lemma 3.18]{LMPTX23} to the setting of g-pairs.
	
	\begin{lem}
		\label{lem:uniruled_nonpsef_can}
		If $(X,B+M)$ is an lc g-pair such that $K_X+B+M$ is not pseudo-effective, then $X$ is uniruled.
	\end{lem}
	
	\begin{proof}
		Suppose that the given g-pair $(X,B+M)$ comes with data $ f \colon W \to X $ and $M_W$. We may assume that $f$ is a log resolution of $(X,B)$ and we may write
		\[ K_W + \Gamma + M_W = f^* (K_X+B+M) + E , \]
		where $\Gamma$ is an effective $\R$-divisor on $W$ with $f_* \Gamma = B$, $E$ is an effective $f$-exceptional $\R$-divisor on $W$, and the $\R$-divisors $\Gamma$ and $E$ have no common components. Then $K_W+\Gamma+M_W$ is not pseudo-effective, since otherwise $ K_X+B+M = f_*(K_W+\Gamma+M_W) $ is also pseudo-effective, a contradiction. Since $\Gamma \geq 0$ and $M_W$ is nef, we deduce that $K_W$ is not pseudo-effective. Therefore, $Y$ is uniruled by \cite[Corollary 0.3]{BDPP13}, and hence $X$ is uniruled.
	\end{proof}
	
	\begin{lem}
		\label{lem:uniruled_psef_antican}
		Let $(X,B+M)$ be a g-pair such that $-(K_X+B+M)$ is pseudo-effective. Assume that $M$ is $\R$-Cartier. Then $X$ is not uniruled if and only if $X$ is canonical with $K_X \sim_\Q 0$, $B=0$ and $M \equiv 0$.
	\end{lem}
	
	\begin{proof}
		Since $M$ is $\R$-Cartier, it is pseudo-effective, and since $-(K_X+B+M)$ is also pseudo-effective by assumption, it is obvious that $(X,B)$ is a pair such that $-(K_X+B)$ is pseudo-effective. Therefore, by \cite[Lemma 3.18]{LMPTX23}, $X$ is not uniruled if and only if $X$ is canonical with $K_X \sim_\Q 0$ and $B=0$, and now the assertion about $M$ follows readily.
	\end{proof}

	\section{Comparison of the asymptotic base loci}
	\label{section:comparison_loci}
	
	Our primary goal in this section is to extend \cite[Proposition 2.8]{BBP13} to the context of generalized pairs. To accomplish this, we follow the strategy of the proof of op.\ cit.\ and we use crucially recent results about the existence of good minimal models of generalized pairs, as well as basic properties of the Nakayama--Zariski decomposition.
	We begin with the following auxiliary result, which is essential for our objective.
	
	\begin{lem}
		\label{lem:loci_non_positive}
		Let $\varphi \colon X \dashrightarrow Y$ be a birational contraction between normal projective varieties and let $D\in\Div_{\R}(X)$ such that $\varphi$ is $D$-non-positive.
		Consider a resolution of indeterminacies $(p,q) \colon W \to X \times Y$ of $\varphi$ such that $ W $ is smooth and write
		\[ p^*D \sim_\R q^* (\varphi_*D) + E \]
		for some effective $q$-exceptional $\R$-divisor $E$ on $W$. If $\varphi_* D$ is semi-ample, then
		\[ \nnef(D) = \dbs(D) = \sbs(D) = p ( \Supp E ) . \]
	\end{lem}
	
	\begin{proof}
		By Remarks \ref{rem:sbs_properties}(3)(4) and \ref{rem:NZD_properties}(3) we deduce that 
		\begin{equation}\label{eq:1_loci_nonpos}
			\sbs (p^* D) = \sbs \big( q^* (\varphi_*D) + E \big) \subseteq \sbs(E) = \Supp E .
		\end{equation}
		By Remark \ref{rem:NZD_properties}(2)(3) we infer that 
		\begin{equation}\label{eq:2_loci_nonpos}
			N_\sigma ( p^* D ) = N_\sigma \big( q^*(\varphi_* D) + E \big) = E .
		\end{equation}
		Since Remark \ref{rem:NZD_properties}(1) implies that 
		$ \Supp \big(N_\sigma ( p^* D )\big) \subseteq \dbs(p^* D) $, 
		by \eqref{eq:1_loci_nonpos} and \eqref{eq:2_loci_nonpos} we have $ \Supp E \subseteq \dbs ( p^* D ) \subseteq \sbs ( p^* D ) \subseteq \Supp E $, and thus, together with \cite[Lemma V.1.9(1)]{Nak04}, we obtain
		\begin{equation}\label{eq:3_loci_nonpos}
			\nnef( p^* D ) =  \dbs ( p^* D ) = \sbs ( p^* D ) = \Supp E .
		\end{equation}
		Therefore, by Remark \ref{rem:sbs_properties}(5), respectively Remark \ref{rem:nnef_properties}(5), and \eqref{eq:3_loci_nonpos} we deduce that
		\[ \sbs(D) = p \big( \sbs(p^*D) \big) = p (\Supp E) \]
		and
		\[ \nnef(D) = p \big( \nnef(p^*D) \big) = p (\Supp E) , \]
		which yield the statement.
	\end{proof}
	
	As an application of Lemma \ref{lem:loci_non_positive} we can show that if $(X,B+M)$ is a g-pair such that $-(K_X+B+M)$ is pseudo-effective and if $X$ is a Mori dream space (e.g., a variety of Fano type), then $\nnef \big( {-}(K_X+B+M) \big)$ is Zariski-closed, cf.\ \cite[Theorem 1.1]{Les14}. Indeed, since $X$ is a Mori dream space, we may run a ${-}(K_X+B+M)$-MMP which terminates with a ${-}(K_X+B+M)$-good minimal model, and thus
	\[ \nnef \big( {-}(K_X+B+M) \big) = \dbs \big( {-}(K_X+B+M) \big) = \sbs \big( {-}(K_X+B+M) \big) \]
	by Lemma \ref{lem:loci_non_positive}, which proves the above assertion. This provides a partial answer to a question raised in \cite[p.\ 155]{CP16}.
	
	In the remainder of this section, Lemma \ref{lem:loci_non_positive} will be applied exclusively to the adjoint divisor of a generalized pair as follows: by Lemma \ref{lem:loci_non_positive} and by \cite[Remark 2.6]{LT22b} we conclude that if $(X,B+M)$ is an lc g-pair such that $K_X+B+M$ is pseudo-effective and if $(X,B+M)$ has a good minimal model, then
	\[ \nnef(K_X+B+M) = \dbs(K_X+B+M) = \sbs (K_X+B+M) . \]
	In particular, the diminished base locus of $K_X+B+M$ is Zariski-closed under these assumptions, cf.\ \cite[Theorem 1.1]{Les14}.
	
	The next theorem is the desired analogue of \cite[Proposition 2.8]{BBP13} in the setting of generalized pairs. Part (iii) is an extension of op.\ cit.\ to the category of generalized pairs, while part (iv) complements part (iii) and is new even for usual pairs. Parts (i) and (ii) constitute generalizations of op.\ cit.\ to the lc setting and are also new for usual pairs.
	
	\begin{thm}
		\label{thm:loci_comparison_can}
		The following statements hold:
		\begin{enumerate}[\normalfont (i)]
			\item If $(X,B+M)$ is an NQC lc g-pair such that $K_X+B+M$ is pseudo-effective, then 
			\[ \nnef(K_X+B+M) = \dbs(K_X+B+M) . \]
			
			\item If $ \big( X,(B+A)+M \big) $ is an NQC lc g-pair such that $K_X+B+A+M$ is pseudo-effective, where $ A $ is an effective ample $\R$-divisor on $X$, then 
			\[ \nnef(K_X+B+A+M) = \dbs(K_X+B+A+M) = \sbs(K_X+B+A+M) . \]
			
			\item If $(X,B+M)$ is a klt g-pair such that $K_X+B+M$ is big, 
			then 
			\[ \nnef(K_X+B+M) = \dbs(K_X+B+M) = \sbs(K_X+B+M) . \]
			
			\item If $(X,B+M)$ is a klt g-pair such that $K_X+B+M$ is pseudo-effective and if $B$ or $M$ is big, then 
			\[ \nnef(K_X+B+M) = \dbs(K_X+B+M) = \sbs(K_X+B+M) . \]
		\end{enumerate}
	\end{thm}
	
	\begin{proof}~
		
		\medskip
		
		\noindent (i) Due to Remark \ref{rem:nnef_properties}(4) it remains to establish the inclusion
		\begin{equation}\label{eq:1_loci_comparison}
			\nnef(K_X+B+M) \supseteq \dbs(K_X+B+M) .
		\end{equation}
		To this end, by the definition of $\dbs$, if $V$ is an irreducible component of $\dbs(K_X+B+M)$, then there exists an ample $\R$-divisor $A$ on $X$ such that $V$ is an irreducible component of $\sbs(K_X+B+M+2A)$. Since again by the definition of $\dbs$ and by Remark \ref{rem:sbs_properties}(3)(4) we have
		\[ \sbs(K_X+B+M+2A) \subseteq \dbs(K_X+B+M+A) \subseteq \dbs(K_X+B+M) , \]
		we infer that $V$ is an irreducible component of $ \dbs(K_X+B+M+A) $. Furthermore, by Remark \ref{rem:avo_properties}(3), we have $\nnef(K_X+B+M+A) \subseteq \nnef(K_X+B+M) $. Consequently, to prove \eqref{eq:1_loci_comparison}, it suffices to show that
		\begin{equation}\label{eq:2_loci_comparison}
			\dbs(K_X+B+M+A) \subseteq \nnef(K_X+B+M+A) .
		\end{equation}
		Taking Remarks \ref{rem:abs+dbs_properties}(1) and \ref{rem:nnef_properties}(1) into account, by replacing $A$ with a general member of its $\R$-linear system, we may assume that $\big(X, (B+A)+M \big)$ is an NQC lc g-pair. Observe also that $K_X+B+A+M$ is big. Therefore, by \cite[Theorem F]{TX23a} and by Lemma \ref{lem:loci_non_positive}, the equality in \eqref{eq:2_loci_comparison} holds, and thus the equality in \eqref{eq:1_loci_comparison} holds, which completes the proof of (i).
		
		\medskip
		
		\noindent (ii) We conclude by \cite[Theorem F]{TX23a} and Lemma \ref{lem:loci_non_positive}.
		
		\medskip
		
		\noindent (iii) We conclude by \cite[Theorem 2.12(ii)]{TX23a} and Lemma \ref{lem:loci_non_positive}.
		
		\medskip
		
		\noindent (iv) We conclude by \cite[Lemma 4.2(ii)]{HanLiu20} and Lemma \ref{lem:loci_non_positive}.
	\end{proof}
	
	\begin{rem}
		\label{rem:loci_comparison_can}
		Let $(X,B+M)$ be a klt (resp.\ NQC lc) g-pair and let $D$ be a pseudo-effective $\R$-divisor on $X$. If there exists $\lambda > 0$ such that $-(K_X+B+M) + \lambda D$ is nef (resp.\ NQC), then setting $L \coloneqq \lambda D - (K_X+B+M) $ we obtain that $\big(X,B+(M+L) \big)$ is a klt (resp.\ NQC lc) g-pair such that $K_X+B+(M+L) = \lambda D$.
			
		In this way, we may easily derive information about the asymptotic base loci $\nnef(D)$, $\dbs(D)$ and $\sbs(D)$ using Theorem \ref{thm:loci_comparison_can}. For instance, if $(X,B+M)$ is an NQC lc g-pair and if $D$ is a pseudo-effective $\R$-divisor on $X$ such that $\lambda D - (K_X+B+M)$ is NQC for some $\lambda > 0$, then we have $ \nnef(D) = \dbs(D) $ by Theorem \ref{thm:loci_comparison_can}(i).
	\end{rem}
	
	We present below some examples to demonstrate that our results in Theorem \ref{thm:loci_comparison_can} are optimal. The first ones show that the inclusions
	\[ \dbs(K_X+B+M) \subseteq \sbs(K_X+B+M) \subseteq \abs(K_X+B+M) \]
	can be strict in the setting of Theorem \ref{thm:loci_comparison_can}(i).
	
	\begin{exa}~
		\label{exa:dbs_vs_sbs_vs_abs}
		
		\medskip
		
		\noindent (i) 
		In this step, following closely \cite[Example 3.14]{HanLiu20}, we construct a surface $X$ with an lc g-pair structure such that the diminished base locus and the stable base locus of the adjoint divisor are different, and the augmented base locus of the adjoint divisor is not the whole surface $X$.
		
		Let $B_0 \subseteq \mathbb{P}^2$ be an elliptic curve and let $\mu \colon X \to \mathbb{P}^2$ be the blow-up of $\mathbb{P}^2$ at $12$ general points $p_1,\dots,p_{12}\in B_0$. Denote by $E_i$ the exceptional prime divisor on $X$ over $p_i$, where $1 \leq i \leq 12$, and set $E \coloneqq \sum_{i=1}^{12} E_i$. Let $H$ be the pullback of a general line in $\mathbb{P}^2$ and let $B \in |3H-E|$ be the strict transform of $B_0$ on $X$. Note that $K_X + B \sim 0$ and $B \cdot E_i = 1$ for every $1 \leq i \leq 12$. By \cite[2.3.A]{Laz04} the divisor $M \coloneqq 4H-E \sim B + H$ is nef and big but not semi-ample, since $B \subseteq \sbs(M)$.
		
		We will determine here the base loci associated with $M$. By Remark \ref{rem:abs+dbs_properties}(3), we have $\dbs(M) = \emptyset$. Next, we will show that $\sbs(M) = \abs(M) = B$. To this end, by Remarks \ref{rem:abs+dbs_properties}(4) and \ref{rem:abs_alt_description}, the locus $\abs(M)$ is the union of finitely many $M$-trivial curves on $X$, one of which is $B$. Assume that there exists an irreducible curve $\ell \subseteq \abs(M)$ which is different from $B$. Then $B \cdot \ell \geq 0$, and since $0 = M \cdot \ell = (B+H) \cdot \ell$ and $H$ is nef, we conclude that $H \cdot \ell = B \cdot \ell =0$. Hence, the curve $\ell$ is contracted by $\mu$, that is, $\ell$ is one of the $E_i$, but we then have $0 = B \cdot \ell = B \cdot E_i = 1$, a contradiction. Consequently, $\abs(M) = B$, which implies that $\sbs(M) = B$.
		
		Observe that $(X,B+M)$ is an lc polarized $\Q$-pair such that $K_X+B+M \sim M$ is nef and big but not semi-ample, and by the previous paragraph we have
		\[ \emptyset = \dbs(K_X+B+M) \subsetneq \sbs(K_X+B+M) = \abs(K_X+B+M) = B \subsetneq X . \]
		
		\noindent (ii) In this step we blow up a point on the surface $X$ constructed in (i) above in order to obtain a surface $Y$ with an lc g-pair structure which further satisfies that the stable base locus and the augmented base locus of the adjoint divisor are different.
		
		Let $\pi \colon Y \to X$ be the blow-up of $X$ at a point $p \not \in B \cup \Supp E$. Denote by $F$ the exceptional divisor of $\pi$ and by $B_Y$ the strict transform of $B$ in $Y$. Set $M_Y \coloneqq \pi^*(M) - F$. We will show that $M_Y$ is nef. To this end, let $\gamma$ be an irreducible curve on $Y$.
		If $\gamma$ is contracted by $\mu\circ\pi$, then $\gamma \in \{\pi^*(E_1),\dots,\pi^*(E_{12}),F\}$, and we have $M_Y\cdot F=1$ and $M_Y \cdot \pi^*(E_i) = B \cdot E_i =1$ for any $1 \leq i \leq 12$. If $\gamma$ is not contracted by $\mu\circ\pi$, then we may write
		\[ \gamma \sim a \pi^*H - \sum_{i=1}^{12}b_i\pi^*E_i - cF, \]
		where $a, b_i, c \in \N$, $ a \geq c$ and $a \geq b_i$ for any $1 \leq i\leq 12$. We compute
		\[ M_Y \cdot \gamma = 4a - \sum_{i=1}^{12} b_i - c . \]
		Since $ B_Y \cdot F = 0$, we have $M_Y \cdot B_Y = M \cdot B = 0$ by Remark \ref{rem:abs_alt_description}. It remains to treat the case when $\gamma \neq B_Y$. We then have $B_Y \cdot \gamma \geq 0$, which implies
		\[ 3a - \sum_{i=1}^{12}b_i \geq 0 . \]
		Hence, 
		\[ M_Y\cdot \gamma \geq  a - c \geq 0 , \]
		which completes the proof of the above assertion. 
		
		Observe that $(Y,B_Y+M_Y)$ is an lc polarized $\Q$-pair such that
		\[ K_Y+B_Y+M_Y \sim \pi^*(M) \sim \pi^* (K_X+B+M) \]
		is nef and big but not semi-ample. Therefore, by Remarks \ref{rem:sbs_properties}(5) and \ref{rem:abs+dbs_properties}(3)(5) we obtain
		\begin{align*}
			\emptyset = \dbs(K_Y+B_Y+M_Y) &\subsetneq \sbs(K_Y+B_Y+M_Y) = B_Y \\
			&\subsetneq \abs(K_Y+B_Y+M_Y) = B_Y \cup F \subsetneq Y .
		\end{align*}
		
		\noindent (iii) In this step we blow up a point on the surface $Y$ constructed in (ii) above in order to obtain a surface $Z$ with an lc g-pair structure which additionally satisfies that the diminished base locus of the adjoint divisor is non-empty.
		
		Let $\varphi \colon Z \to Y$ be the blow-up of $Y$ at some point and denote by $G$ the reduced exceptional divisor of $\varphi$. Let $C_Z$ (resp.\ $F_Z$) be the strict transform of $B_Y$ (resp.\ $F$) on $Z$ and set $B_Z \coloneqq \varphi^*B_Y$ and $M_Z \coloneqq \varphi^* M_Y$. Then $(Z,B_Z+M_Z)$ is a lc polarized $\Q$-pair such that
		\[ K_Z+B_Z+M_Z\sim \varphi^*(K_Y+B_Y+M_Y)+G \]
		is big but not nef, since $(K_Z+B_Z+M_Z)\cdot G=-1$. We will show in the next paragraph that 
		\begin{align*}
			\emptyset &\subsetneq \dbs(K_Z+B_Z+M_Z) = G \\ 
			&\subsetneq \sbs(K_Z+B_Z+M_Z) = C_Z \cup G \\ 
			&\subsetneq \abs(K_Z+B_Z+M_Z) = C_Z \cup F_Z \cup G \subsetneq Z .
		\end{align*}
		
		First, since $N_\sigma(K_Z+B_Z+M_Z)=G$ by Remark \ref{rem:NZD_properties}(2)(3), it follows from Remarks \ref{rem:abs+dbs_properties}(3)(4) and \ref{rem:NZD_properties}(1) that 
		\[ \dbs(K_Z+B_Z+M_Z) = G . \]
		Second, by Remark \ref{rem:abs+dbs_properties}(5) we obtain 
		\[ \abs(K_Z+B_Z+M_Z) = C_Z \cup F_Z \cup G . \]
		It remains to show that
		\[ \sbs(K_Z+B_Z+M_Z) = C_Z \cup G . \]
		To this end, by Remark \ref{rem:sbs_properties}(5) we obtain
		\[ C_Z \subseteq \varphi^{-1} \big( \sbs(K_Y+B_Y+M_Y) \big) = \sbs \big( \varphi^* (K_Y+B_Y+M_Y) \big).\] 
		Since $G \subseteq \sbs(K_Z+B_Z+M_Z)$, for any $m\in\N$ it follows that
		\begin{align*}
			|m(K_Z+B_Z+M_Z)| &= |m\varphi^* (K_Y+B_Y+M_Y)+mG| \\
			&= |m\varphi^* (K_Y+B_Y+M_Y)| + mG .
		\end{align*}
		Hence, taking Remark \ref{rem:sbs_properties}(1) into account, we obtain
		\[ \sbs(K_Z+B_Z+M_Z)= \sbs \big( \varphi^* (K_Y+B_Y+M_Y) \big) \cup G = C_Z \cup G , \]
		as claimed.
	\end{exa}
	
	The next example indicates that the inclusion
	\[ \sbs(K_X+B+M) \subseteq \abs(K_X+B+M) \]
	can be strict in the setting of Theorem \ref{thm:loci_comparison_can}(iii).
	
	\begin{exa}
		Let $(X,B)$ be a klt $\Q$-pair such that $-(K_X+B)$ is nef and big, but not ample. Set $M \coloneqq -2(K_X+B)$. Then $(X,B+M)$ is a klt polarized $\Q$-pair such that $K_X+B+M = -(K_X+B)$ is big and semi-ample by \cite[Theorem 3.3]{KM98}, but not ample. Therefore,
		\[ \emptyset = \dbs(K_X+B+M) = \sbs(K_X+B+M) \subsetneq \abs(K_X+B+M) \subsetneq X . \]
		We discuss in the next paragraph a concrete example in the 
		present setting.
		
		Let $X$ be a weak del Pezzo surface which is not del Pezzo, i.e., $-K_X$ is nef and big, but not ample. For example, take $X$ to be the blow-up of $\mathbb{P}^2$ at three points lying on a line. Set $B \coloneqq 0$ and $M \coloneqq -2K_X$. Then $K_X+B+M = -K_X$ is big and semi-ample, but not ample. By \cite[Proposition 4]{LO16} there are finitely many $K_X$-trivial irreducible curves $C_1, \dots, C_k$ on the surface $X$, which are smooth rational curves; see \cite[Exercises IV.1.8(b) and V.1.3(a)]{Har77}. Therefore, taking Remark \ref{rem:abs_alt_description} into account, we obtain
		\begin{align*}
			\emptyset &= \dbs(K_X+B+M) = \sbs(K_X+B+M) \\ 
			&\subsetneq \abs(K_X+B+M)= C_1 \cup \ldots \cup C_k \subsetneq X .
		\end{align*}
	\end{exa}
	
	We provide in Corollaries \ref{cor:loci_comparison_arb_nef_antican}, \ref{cor:loci_comparison_arb_lc_partial} and \ref{cor:loci_comparison_abundant}(ii) some conditions under which the non-nef locus and the diminished base locus of a pseudo-effective $\R$-divisor $D$ on a normal projective variety coincide, cf.\ \cite[Theorem 1.2]{CDB13}. We emphasize that Corollary \ref{cor:loci_comparison_arb_lc_partial} establishes some new special cases of Boucksom, Broustet and Pacienza's conjecture \cite[Conjecture 2.7]{BBP13} in the lc setting.
	
	\begin{cor}
		\label{cor:loci_comparison_arb_nef_antican}
		Let $(X,B+M)$ be a klt g-pair such that $K_X+B+M \equiv 0$ and let $D\in\Div_{\R}(X)$. The following statements hold:
		\begin{enumerate}[\normalfont (i)]
			\item If $D$ is pseudo-effective, then
			\[ \nnef(D) = \dbs(D). \]
			
			\item If $D$ is big and if $K_X+B+M \sim_\R 0$, then 
			\[ \nnef(D) = \dbs(D) = \sbs(D) . \]
			
			\item Assume that $|D|_\R \neq \emptyset$. If $B$ or $M$ is big and if $K_X+B+M \sim_\R 0$, then
			\[ \nnef(D) = \dbs(D) = \sbs(D) . \]
		\end{enumerate}
	\end{cor}
	
	\begin{proof}~
		
		\medskip
		
		\noindent (i) To prove the assertion, by arguing as in the proof of Theorem \ref{thm:loci_comparison_can}(i), we see that it suffices to show that $ \dbs(D+A) \subseteq \nnef(D+A) $ for some ample $\R$-divisor $A$ on $X$ chosen as in that proof. To this end, in view of Remarks \ref{rem:abs+dbs_properties}(1)(2), \ref{rem:nnef_properties}(1)(2), \cite[Remark 4.2(2)]{BZ16} and the fact that $D+A$ is big, by first replacing $D+A$ with a member of its $\R$-linear system and then with $\varepsilon (D+A)$ for some sufficiently small $\varepsilon > 0$, we may assume that $\big( X, (B+D+A) +M \big)$ is a klt g-pair such that $K_X+B+D+A+M \equiv D+A$. Hence, (i) follows from Theorem \ref{thm:loci_comparison_can}(iii) or \ref{thm:loci_comparison_can}(iv).
		
		\medskip
		
		\noindent (ii) By first replacing $D$ with a member of its $\R$-linear system and then with $\varepsilon D$ for some sufficiently small $\varepsilon > 0$, we may assume that $\big( X,(B+D)+M \big)$ is a klt g-pair such that $K_X+B+D+M \sim_\R D$, so (ii) follows from Theorem \ref{thm:loci_comparison_can}(iii), taking Remarks \ref{rem:sbs_properties}(2), \ref{rem:abs+dbs_properties}(1)(2) and \ref{rem:nnef_properties}(1)(2) into account.
		
		\medskip
		
		\noindent (iii) We argue as in the proof of (ii), except that we now invoke Theorem \ref{thm:loci_comparison_can}(iv) instead of Theorem \ref{thm:loci_comparison_can}(iii).
	\end{proof}
	
	The next examples demonstrate that the equality $\dbs(D) = \sbs(D)$ in Corollary \ref{cor:loci_comparison_arb_nef_antican}(ii) need not hold in general if we drop the assumption that $(X,B+M)$ has klt singularities (cf.\ Corollary \ref{cor:loci_comparison_arb_lc_partial}) or that $D$ is big (cf.\ Corollary \ref{cor:loci_comparison_abundant}(ii)).
	
	\begin{exa}~
		
		\medskip
		
		\noindent (i) \cite[Example 1.1]{Sho00}:
		Let $S \coloneqq \mathbb{P}(\mathcal{E}) \to C$ be a $\mathbb{P}^1$-bundle over an elliptic curve $C$, where $\mathcal{E}$ is a rank-two vector bundle over $C$ which is defined by an non-split extension 
		\[ 0 \to \mathcal{O}_C \to \mathcal{E} \to \mathcal{O}_C \to 0 . \]
		Then $-K_S$ is nef with $\kappa(S,-K_S) = 0$ and the unique member in the linear system $|{-}K_S|$ is $2C_0$, where 
		$ \mathcal{O}_S(C_0)\simeq \mathcal{O}_{\mathbb{P}(\mathcal{E})}(1) $. 
		Therefore,
		\[ \emptyset = \dbs(C_0) \subsetneq \sbs(C_0) = C_0 \subsetneq \abs(C_0) = S . \]
		Here, we view $(S,2C_0)$ either as a klt polarized $\Q$-pair with boundary part $B = 0$ and nef part $M = 2C_0$ or as an lc polarized $\Q$-pair with boundary part $B = C_0$ and nef part $M = C_0$, which satisfies $K_S+B+M \sim 0$.
		
		\medskip
		
		\noindent (ii) Consider the rank-two vector bundle $\mathcal{V}\coloneqq \mathcal{O}_S\oplus \mathcal{O}_S(-2C_0-\ell_0)$ on $S$, its projectivization $ \pi \colon X \coloneqq \mathbb{P} (\mathcal{V}) \to S $, and denote by $
		\gamma$ a fiber of $\pi$. Set 
		$ E\coloneqq \mathbb{P}\big(\mathcal{O}_S(-2C_0-\ell_0)\big)\simeq S $,
		note that 
		$\mathcal{O}_X(E)\simeq\mathcal{O}_{\mathbb{P}(\mathcal{V})}(1)$, and denote by $\ell_E$ a fibre of the ruled surface $E$ and by $C_E\subseteq E$ the zero section. We have
		\begin{equation}\label{eq:exa_can_bundle}
			\omega_X \simeq \pi^*(\omega_S\otimes\det\mathcal{V}) \otimes \mathcal{O}_{\mathbb{P}(\mathcal{V})}(-2)\simeq \pi^*\mathcal{O}_S(-4C_0-\ell_0)\otimes \mathcal{O}_X(-2E).\
		\end{equation}
		By Grothendieck's relation \cite[Appendix A, Section 3]{Har77} we obtain $E^2 = E\cdot \pi^*\big(c_1(\mathcal{V})\big)$, and thus
		\begin{equation}\label{eq:exa_Groth_rel}
			E|_E\equiv -2C_E-\ell_E.
		\end{equation}
		It is easy to check that the three extremal rays of the Mori cone of $X$ are generated by the class of $\ell_E$, the class of $C_E$ and the class of $
		\gamma$, respectively. By \eqref{eq:exa_can_bundle} and \eqref{eq:exa_Groth_rel}, we obtain
		\[ -(K_X+E) \cdot \ell_E =2, \quad -(K_X+E) \cdot C_E = 0, \quad -(K_X+E) \cdot \gamma = 1. \]
		Therefore, $-(K_X+E)$ is nef.
		
		Set $ L \coloneqq -(K_X+E)$ and consider the lc polarized $\Q$-pair $(X,E+L)$ and the $\Z$-divisor $D \coloneqq E+\pi^*C_0$ on $X$. Since 
		$$D \cdot \ell_E = E \cdot \ell_E +(\pi^*C_0)\cdot \ell_E =-1$$
		and $\pi^*C_0$ is nef, we deduce that $\dbs(D)=E$. Since 
		\[ \mathcal{O}_X(D) \simeq \mathcal{O}_{\mathbb{P}(\mathcal{V})}(1)\otimes \pi^* \OO_S(C_0) \]
		and $\kappa(S,C_0)=0$, for any $m\in\N$ we have
		\begin{align*}
			h^0\big(X,\mathcal{O}_X(mD)\big) &= h^0\big(S,S^m\big(\mathcal{V}\otimes\mathcal{O}_S(C_0)\big)\big) \\ 
			&= h^0\big(S,S^m\big(\mathcal{O}_S(C_0)\oplus\mathcal{O}_S(-C_0-\ell_0)\big)\big) = 1 , 
		\end{align*}
		so $\kappa(X,D)=0$. In conclusion, we have
		\[ \emptyset \subsetneq \dbs(D) = E \subsetneq \sbs(D) = D \subsetneq \abs(D) = X . \]
	\end{exa}
	
	\begin{cor}
		\label{cor:loci_comparison_arb_lc_partial}
		Let $(X,B+M)$ be an NQC lc g-pair such that $K_X+B+M \equiv 0$ and let $D\in\Div_{\R}(D)$ with $|D|_\R \neq \emptyset$. If there exists $G \in |D|_\R$ such that 
		\[ \tau \coloneqq \sup \{ t \in [0,1] \mid K_X+ (B+tG) + M \emph{ is lc} \} > 0 , \]
		then
		\[ \nnef(D) = \dbs(D) . \]
		If, additionally, $K_X+B+M \sim_\R 0$ and if there exists an ample $\R$-divisor $A$ on $X$ with $B \geq A \geq 0$, then
		\[ \nnef(D) = \dbs(D) = \sbs(D) . \]
	\end{cor}
	
	\begin{proof}
		Note that $\big(X, (B + \tau G) + M \big)$ is an NQC lc g-pair such that $ K_X + B + \tau G + M \equiv \tau D $ (resp.\ $ K_X + B + \tau G + M \sim_\R \tau D $), so the first (resp.\ second) part of the statement follows from Theorem \ref{thm:loci_comparison_can}(i) (resp.\ Theorem \ref{thm:loci_comparison_can}(ii)), taking Remarks \ref{rem:abs+dbs_properties}(1)(2) and \ref{rem:nnef_properties}(1)(2) into account.
	\end{proof}
	
	With the aid of Theorem \ref{thm:EGMM_klt_abundant} we can refine both Theorem \ref{thm:loci_comparison_can}(iii) and Corollary \ref{cor:loci_comparison_arb_nef_antican}(ii) by replacing the assumption that the given $\R$-divisor $D$ is big with the weaker hypothesis that $D$ is abundant.
	
	\begin{cor}
		\label{cor:loci_comparison_abundant}
		Let $(X,B+M)$ be an NQC klt g-pair. The following statements hold:
		\begin{enumerate}[\normalfont (i)]
			\item If $K_X+B+M$ is pseudo-effective and abundant, then 
			\[ \nnef(K_X+B+M) = \dbs(K_X+B+M) = \sbs(K_X+B+M) . \]
			
			\item If $K_X+B+M \sim_\R 0$ and if $D$ is a pseudo-effective and abundant $\R$-divisor on $X$, then
			\[ \nnef(D) = \dbs(D) = \sbs(D) . \]
		\end{enumerate}
	\end{cor}
	
	\begin{proof}~
		
		\medskip
		
		\noindent (i) We conclude by Theorem \ref{thm:EGMM_klt_abundant} and Lemma \ref{lem:loci_non_positive}.
		
		\medskip
		
		\noindent (ii) Arguing as in the proof of Corollary \ref{cor:loci_comparison_arb_nef_antican}(ii), we may assume that $\big( X,(B+D)+M \big)$ is a klt g-pair such that $K_X+B+D+M \sim_\R D$. We conclude by (i).
	\end{proof}
	
	By combining Lemma \ref{lem:reduction_to_Q-pairs}(i) and \cite[Theorem 1.2]{CDB13}, we can readily extend the latter to the category of g-pairs: if $(X,B+M)$ is a klt g-pair and if $D$ is a pseudo-effective $\R$-divisor on $X$, then $ \nnef(D) = \dbs(D) $. We recall that the proof of \cite[Theorem 1.2]{CDB13} relies on the theory of multiplier ideals, whereas our Corollaries \ref{cor:loci_comparison_arb_nef_antican}, \ref{cor:loci_comparison_arb_lc_partial} and \ref{cor:loci_comparison_abundant}(ii) are consequences of the existence of good minimal models for certain classes of g-pairs.
	
	\medskip
	
	We conclude this section by discussing an application of Corollary \ref{cor:loci_comparison_abundant}(i) to the study of strictly nef divisors. Recall that a divisor $D\in\Div_{\Q}(X)$ on a projective variety $X$ is called \emph{strictly nef} if $D \cdot C > 0$ for every curve $C$ on $X$. Note that an ample divisor is strictly nef, but the converse does not necessarily hold, as demonstrated by Mumford's example \cite[Chapter I, Example 10.6]{Har70}. On the other hand, we have the following conjecture of Beltrametti and Sommese \cite{BelSom94} about the positivity of strictly nef divisors.
	
	\begin{conj}
		Let $X$ be a smooth projective variety and let $D$ be an effective strictly nef divisor on $X$. If $D-K_X$ is nef, then $D$ is ample.
	\end{conj}
	
	It follows from \cite[Lemma 1.3]{Ser95} that the above conjecture holds for strictly nef divisors which are \emph{big}. We now prove a version of Beltrametti and Sommese's conjecture in the singular setting for strictly nef $\Q$-divisors which are \emph{abundant}.
	
	\begin{prop}
		\label{prop:strictly_nef_abundant}
		Let $(X,B+M)$ be a klt $\Q$-g-pair and let $D$ be an abundant strictly nef $\Q$-divisor on $X$. If $D - (K_X+B+M)$ is nef, then $D$ is ample.
	\end{prop}
	
	\begin{proof}
		Setting $ L \coloneqq D - (K_X+B+M) $, we obtain a klt $\Q$-g-pair $ \big(X,B+(M+L)\big) $ such that $K_X+B+M+L = D$ is strictly nef and abundant. Therefore, by Remark \ref{rem:abs+dbs_properties}(3) and Corollary \ref{cor:loci_comparison_abundant}(i) we obtain $ \dbs(D) = \sbs(D) = \emptyset $, so $D$ is semi-ample by Remark \ref{rem:sbs_properties}(4). Thus, the assertion follows from \cite[Lemma 1.4]{CCP08}.
	\end{proof}
	
	We would like to thank P.\ Chaudhuri for informing us about an alternative proof of Proposition \ref{prop:strictly_nef_abundant} which uses \cite[Theorem 2 and Lemma 6]{Chaud23b} to show that $D$ is semi-ample.

	\section{On the uniruledness of the asymptotic base loci: the klt case}
	\label{section:uniruledness_loci}
	
	Our first objective in this section is to prove Theorem \ref{thm:loci_uniruledness_can}. To this end, we first prove two auxiliary results, namely Lemmata \ref{lem:MMP_step_abs} and \ref{lem:exc_locus_uniruled}.
	
	\begin{lem}
		\label{lem:MMP_step_abs}
		Let $(X,B+M)$ be an NQC klt g-pair such that $K_X+B+M$ is big but not nef. Consider a step of a $(K_X+B+M)$-MMP:
		\begin{center}
			\begin{tikzcd}
				(X,B+M) \arrow[rr, dashed, "\varphi"] \arrow[dr, "g" swap] && (X',B'+M') \arrow[dl, "h"] \\
				& Y
			\end{tikzcd}
		\end{center}
		Let $V$ be an irreducible component of $\abs(K_X+B+M)$. The following statements hold:
		\begin{enumerate}[\normalfont (i)]
			\item If $V$ is contained in $\Exc(g)$, then $V$ is uniruled.
			
			\item If $V$ is not contained in $\Exc(g)$, then the strict transform $V'$ of $V$ on $X'$ is an irreducible component of $\abs(K_{X'}+B'+M')$.
		\end{enumerate}
	\end{lem}
	
	\begin{proof}~
		
		\medskip
		
		\noindent
		(i) First, we claim that $\Exc(g) \subseteq \abs(K_X+B+M)$. Indeed, pick a curve $C \subseteq X$ contracted by $g$ and fix an element $D \in |K_X+B+M|_\R$. Then $(K_X+B+M) \cdot C < 0$, so $D \cdot C < 0$, which implies that $C \subseteq \Supp D.$ Therefore, \[C \subseteq \sbs(K_X+B+M) \subseteq \abs(K_X+B+M),\] and the above assertion follows. Now, by assumption and by the claim we infer that $V$ is an irreducible component of $\Exc(g)$, so it is uniruled by Lemma \ref{lem:reduction_to_Q-pairs}(ii).
		
		\medskip
		
		\noindent (ii) 
		By assumption there exists a resolution of indeterminacies $(p,q) \colon W \to X \times X'$ of $\varphi$ such that $p$ (resp.\ $q$) is an isomorphism over the generic point of $V$ (resp.\ $V'$). 
		\begin{center}
			\begin{tikzcd}
				& W \arrow[dr, "q"] \arrow[dl, "p" swap] \\
				X \arrow[rr, "\varphi", dashed] && X'
			\end{tikzcd}
		\end{center}
		It follows from \cite[Lemma 2.8(ii)]{LMT23} that we may write 
		\[ p^* (K_X+B+M) \sim_\R q^* (K_{X'}+B'+M') + E , \]
		where $E$ is an effective $q$-exceptional $\R$-Cartier $\R$-divisor on $W$, so by Remark \ref{rem:abs+dbs_properties}(5) we obtain
		\begin{align*}
			p^{-1} \big( \abs(K_X+B+M) \big) \cup \Exc(p) &= \abs \big( p^*(K_X+B+M) \big) \\ 
			&= \abs \big( q^* (K_{X'}+B'+M') + E \big) \\
			&= q^{-1} \big( \abs(K_{X'}+B'+M') \big) \cup \Exc(q) .
		\end{align*}
		
		Set $V_W \coloneqq p^{-1}_* V = q^{-1}_* V'$ and observe that $V_W \subseteq p^{-1} \big( \abs(K_X+B+M) \big)$. Note that the generic point of $V_W$ does not belong to either $\Exc(p)$ or $\Exc(q)$ by construction. Therefore, $ V_W $ is an irreducible subset of $ q^{-1} \big( \abs(K_{X'}+B'+M') \big) $, and thus $V' $ is an irreducible subset of $ \abs (K_{X'}+B'+M')$. Since $V$ is an irreducible component of $\abs(K_X+B+M)$, we infer that $V_W$ is an irreducible component of $p^{-1} \big( \abs(K_X+B+M) \big)$, and thus $V_W$ is an irreducible component of  $q^{-1} \big( \abs(K_{X'}+B'+M') \big)$. Consequently, $V'$ is an irreducible component of $\abs(K_{X'}+B'+M')$, as asserted.
	\end{proof}
	
	\begin{lem}
		\label{lem:exc_locus_uniruled}
		Let $(X,B+M)$ be an NQC klt g-pair. Assume that there exists a projective birational morphism $f \colon X \to Y$ with connected fibers to a normal projective variety $Y$ such that $K_X+B+M \sim_{\R,Y} 0$. Then every irreducible component of $\Exc(f)$ is uniruled.
	\end{lem}
	
	\begin{proof}
		By the canonical bundle formula \cite{Fil20,HanLiu21} there exists an NQC klt g-pair structure $ (Y,B_Y+M_Y) $ on $ Y $ such that
		\[ K_X+B+M \sim_\R f^*(K_Y+B_Y+M_Y) , \]
		so there exists a klt $\Q$-pair structure $(Y,\Delta_Y)$ on $Y$ by Lemma \ref{lem:reduction_to_Q-pairs}(i). It follows from \cite[Corollary 1.5(1)]{HM07a} and \cite[Proposition IV.3.3(4)]{Kol96} that the positive-dimensional fibers of $f$ are uniruled.
		
		Consider an irreducible component $V$ of $\Exc(f)$ and fix a general point $x \in V$, that is, a point of $V$ which does not belong to any other irreducible component of $\Exc(f)$. Then the fiber $f^{-1}\big(f(x)\big)$ is positive-dimensional, and hence uniruled. Therefore, there exists an irreducible rational curve $C \subseteq f^{-1}\big(f(x)\big)$ passing through $x$, and thus $C \subseteq V$. Consequently, $V$ is uniruled.
	\end{proof}
	
	The following example demonstrates that neither \cite[Theorem 2]{Kaw91} nor \cite[Corollary 1.5(1)]{HM07a} hold in the strictly lc setting.
	
	\begin{exa}[{\cite[Example 6.4]{Tak08}}]
		\label{exa:blowup_vertex_of_cone}
		Let $S \subseteq \mathbb{P}^3$ be a projective cone over an elliptic curve $C$ and denote by $H$ a hyperplane section of $S$. Let $\mu \colon X\to S$ be the blow-up of $S$ at the vertex $v \in S$ and denote by $E$ the exceptional divisor of $\mu$. Then $(X,E)$ is a strictly lc pair such that $ -(K_X+E) = \mu^*H $ is big and semi-ample. Note also that $X$ is smooth and a $\mathbb{P}^1$-bundle over $C$, so it is uniruled, whereas $E = \mu^{-1}(v) \cong C$ is not uniruled.
	\end{exa}
	
	As a corollary of Lemma \ref{lem:exc_locus_uniruled}, we give an affirmative answer to \cite[Question 6.4]{Bir17} (over an algebraically closed field of characteristic $0$) in a more general framework.
	
	\begin{cor}
		\label{cor:Bir17_Q}
		If $(X,B+M)$ is an NQC klt g-pair such that $K_X+B+M$ is nef and big, then  $\abs(K_X+B+M)$ is covered by rational curves $C$ such that $(K_X+B+M) \cdot C = 0$.
	\end{cor}
	
	\begin{proof}
		By assumption and by \cite[Lemma 2.10 and Theorem 2.12(ii)]{TX23a} we infer that $K_X+B+M$ is semi-ample. Therefore, there exists a birational fibration $f \colon X \to Y$ such that $K_X+B+M \sim_\R f^* A$, where $A$ is an ample $\R$-divisor on $Y$, so by Remarks \ref{rem:sbs_properties}(4) and \ref{rem:abs+dbs_properties}(5) we deduce that $ \abs(K_X+B+M) = \Exc(f) $. Hence, Lemma \ref{lem:exc_locus_uniruled} implies that the irreducible components of $\Exc(f)$ are covered by rational curves $C$ which are contracted by $f$, and thus $(K_X+B+M) \cdot C = 0$.
	\end{proof}
	
	Using Remark \ref{rem:loci_comparison_can} and Corollary \ref{cor:Bir17_Q}, we can also deduce readily the following similar statement in the dual setting, cf.\ \cite[Theorem 1.3]{Tak08}.
	
	\begin{cor}
		\label{cor:abs_uniruledness_weak_Fano}
		If $(X,B+M)$ is an NQC klt g-pair such that ${-}(K_X+B+M)$ is nef and big, then the irreducible components of $\abs \big( {-}(K_X+B+M) \big)$ are covered by rational curves $C$ such that ${-}(K_X+B+M) \cdot C = 0$.
	\end{cor}
	
	We now prove Theorem \ref{thm:loci_uniruledness_can}.
	
	\begin{proof}[Proof of Theorem \ref{thm:loci_uniruledness_can}]~
		
		\medskip
		
		\noindent (i) The equality $ \nnef(K_X+B+M) = \dbs(K_X+B+M) $ was established in Theorem \ref{thm:loci_comparison_can}(i), so it remains to prove the uniruledness of the components of this locus. If $V$ is an irreducible component of $\dbs(K_X+B+M)$, then by \cite[Lemma 1.14]{ELMNP06} we have
		\[ \dbs(K_X+B+M) = \bigcup_H \abs(K_X+B+M+H) , \]
		where the union is taken over all ample $\R$-divisors $H$ on $X$, so there exists an ample $\R$-divisor $A$ on $X$ such that $V$ is an irreducible component of $\abs(K_X+B+M+A)$. In view of Remark \ref{rem:abs+dbs_properties}(1), by replacing $A$ with a general member of its $\R$-linear system, we may assume that $\big(X, (B+A)+M \big)$ is an NQC klt g-pair, and we also observe that $K_X+B+A+M$ is big. Therefore, the claimed uniruledness of $V$ follows from the second part of (ii), which will be proved afterwards.
		
		\medskip
		
		\noindent (ii) We will first settle the second part of the statement, that is, we will show that every irreducible component $V$ of $\abs(K_X+B+M)$ is uniruled. To this end, by \cite[Lemma 2.10, Theorem 2.12(ii), and Theorem 4.2]{TX23a} we may run a $ (K_X+B+M)$-MMP with scaling of an ample divisor which terminates after $r-1$ steps (where $r \geq 1$) with a good minimal model $ (Y,B_Y+M_Y)$ of $(X,B+M)$. We denote by 
		\begin{center}
			\begin{tikzcd}
				(X_i,B_i+M_i) \arrow[rr, dashed, "\varphi_i"] \arrow[dr, "g_i" swap] && (X_{i+1},B_{i+1}+M_{i+1}) \arrow[dl, "h_i"] \\
				& Z_i
			\end{tikzcd}
		\end{center}
		the $i$-th step of this MMP, where $i \in \{1, \dots, r-1\}$, $X_1 \coloneqq X$ and $X_r \coloneqq Y$. We distinguish two cases. If there exists an index $i \in \{1, \dots, r-1\}$ such that the strict transform $V_i$ of $V$ on $X_i$ is contained in $\Exc(g_i)$, then $V_i$ is uniruled by Lemma \ref{lem:MMP_step_abs}(i), so $V$ itself is uniruled. Otherwise, Lemma \ref{lem:MMP_step_abs}(ii) implies that the strict transform $V_Y$ of $V$ on $Y$ is an irreducible component of $\abs(K_Y+B_Y+M_Y)$. Since $K_Y+B_Y+M_Y$ is big and semi-ample, there exists a birational fibration $\psi \colon Y \to Z$ such that $K_Y+B_Y+M_Y \sim_\R \psi^* A_Z$, where $A_Z$ is an ample $\R$-divisor on $Z$, and it follows readily from Remarks \ref{rem:sbs_properties}(4) and \ref{rem:abs+dbs_properties}(5) that $ \abs(K_Y+B_Y+M_Y) = \Exc(\psi) $. Therefore, $V_Y$ is an irreducible component of $\Exc(\psi)$, so it is uniruled by Lemma \ref{lem:exc_locus_uniruled}, and thus $V$ itself is uniruled. This completes the proof of the second part of (ii), and hence the proof of (i) as well.
		
		Finally, the first part of (ii) follows immediately from (i) and Theorem \ref{thm:loci_comparison_can}(iii).
		
		\medskip
		
		\noindent (iii) We conclude by (i) and Theorem \ref{thm:loci_comparison_can}(iv).
	\end{proof}
	
	We present below various examples in order to clarify the cases that are not treated by Theorem \ref{thm:loci_uniruledness_can}.
	
	In the setting of Theorem \ref{thm:loci_uniruledness_can}(i) we observe that for any $D\in\Div_{\R}(X)$ which is not big we have $\abs(D) = X$. An example in this context such that $X$ is not uniruled can be found in \cite[p.\ 212, Abundance]{BH14b}; see also Example \ref{exa:uniruledness_fails_II}(ii). The next example indicates that in general the irreducible components of $\sbs(D)$ need not be uniruled either, unless we have stronger hypotheses on both $K_X+B+M$ and $D$; see for instance Corollary \ref{cor:loci_uniruledness_abundant}(ii).
	
	\begin{exa}
		\label{exa:uniruledness_fails_I}
		Let $X \to \mathbb{P}^3$ be the blow-up of $\mathbb{P}^3$ at eight very general points. By \cite[Section 2]{LO16}, $-K_X$ is nef, but neither big nor semi-ample, and it is divisible by $2$ in $\Pic(X)$. If $S$ is a general member of the linear system $|{-}\frac{1}{2}K_X|$, then $S$ is a smooth projective surface isomorphic to $\mathbb{P}^2$ blown up at nine very general points. Moreover, the base locus of the linear system $|{-}K_X|$ is an elliptic curve, denoted by $C$. It follows from Remark \ref{rem:sbs_properties}(1) that $\sbs(-K_X) = C$. Now, set $B \coloneqq \varepsilon S$ for some $\varepsilon \in \Q \cap(0,1)$ and $M \coloneqq -K_X$. Then $(X,B+M)$ is a klt polarized $\Q$-pair such that $K_X+B+M = \varepsilon S $, and we have
		\begin{align*}
			\emptyset = \dbs(K_X+B+M) &\subsetneq \sbs(K_X+B+M) = C \\ 
			&\subsetneq \abs(K_X+B+M) = X .
		\end{align*}
		In particular, $\sbs(K_X+B+M)$ is not uniruled, while $\abs(K_X+B+M) = X$ is uniruled.
	\end{exa}
	
	The two examples below demonstrate that the analogue of Theorem \ref{thm:loci_uniruledness_can}(ii) in the strictly lc setting fails in general.
	
	\begin{exa}~
		
		\medskip
		
		\noindent (i) Let $\mu \colon X \to S$, $E$ and $H$ be as in Example \ref{exa:blowup_vertex_of_cone}. Set $B \coloneqq E$ and $M \coloneqq 2\mu^*H$. Then $(X,B+M)$ is an lc polarized $\Q$-pair which is not klt and whose canonical class $K_X+B+M = \mu^*H $ is big and semi-ample. Therefore, 
		\[ \nnef(K_X+B+M) = \dbs(K_X+B+M) = \sbs(K_X+B+M) = \emptyset , \]
		while 
		\[ \abs(K_X+B+M) = \Exc(\mu) = E \]
		is not uniruled.
		
		\medskip
		
		\noindent (ii) Consider the polarized $\Q$-pair $(Z,B_Z+M_Z)$ from Example \ref{exa:dbs_vs_sbs_vs_abs}(iii), which is lc but not klt and whose canonical class $K_Z+B_Z+M_Z$ is big but not nef. Observe that both loci $\sbs(K_Z+B_Z+M_Z)$ and $\abs(K_Z+B_Z+M_Z)$ have an irreducible component which is not uniruled.
	\end{exa}
	
	We now briefly discuss the complementary setting of Theorem \ref{thm:loci_uniruledness_can}. Given an NQC klt g-pair $(X,B+M)$ such that $K_X+B+M$ is not pseudo-effective, we consider a pseudo-effective $\R$-divisor $D$ on $X$ and we distinguish two cases. If $D$ is not big, then $\abs(D) = X$ is uniruled by Lemma \ref{lem:uniruled_nonpsef_can}, but the irreducible components of $\dbs(D)$ or $\sbs(D)$ need not be uniruled in general. If $D$ is big, then the same holds for all asymptotic base loci of $D$. The following examples illustrate these phenomena.
	
	\begin{exa}~
		
		\medskip
		
		\noindent (i) Let $C$, $\mu \colon X \to S$, $H$ and $E$ be as in Example \ref{exa:blowup_vertex_of_cone}. Set $B \coloneqq 0$ and $M \coloneqq \mu^* H$. Then $(X,B+M)$ is a klt polarized $\Q$-pair such that $-(K_X+B+M) = E \geq 0$ is neither nef nor big, since $E^2=-1$ and $\kappa(X,E)=0$. In particular, $K_X+B+M$ is not pseudo-effective. By Remark \ref{rem:NZD_properties}(3) we obtain $ \dbs(E) = \sbs(E) = E \cong C $, which is not uniruled, while $ \abs(E) = X $ is uniruled.
		
		\medskip
		
		\noindent (ii) Let $B_0 \subseteq \mathbb{P}^2$ be an elliptic curve and let $\mu \colon X \to \mathbb{P}^2$ be the blow-up of $\mathbb{P}^2$ at $13$ general points $ p_1, \dots, p_{13} \in B_0$. Denote by $E_i$ the exceptional prime divisor on $X$ over $p_i$, where $1 \leq i \leq 13$, and set $E \coloneqq \sum_{i=1}^{13} E_i$. Let $H$ be the pullback of a general line in $\mathbb{P}^2$ and let $B \in |3H-E|$ be the strict transform of $B_0$ on $X$. Consider the smooth pair $(X,0)$ and note that $K_X$ is not pseudo-effective. Set $ D \coloneqq H+B \sim 4H-E $ and observe that $D$ is big but not nef, as $D \cdot B = -1$. We will determine below the base loci associated with $D$.
		
		Since $D \cdot B < 0$, by Remark \ref{rem:dbs_negative_intersection} we infer that $B \subseteq \dbs(D)$, while by Remark \ref{rem:sbs_properties}(1)(3)(4) and since $B \geq 0$, we deduce that $\sbs(D) \subseteq B$. Thus, $ \dbs(D) = \sbs(D) = B $. Finally, to compute $\abs(D)$, we will first determine the Zariski decomposition $P_\sigma(D) + N_\sigma(D)$ of $D$, since then by \cite[Example 1.11]{ELMNP06} we obtain $\abs(D)=\abs\big(P_{\sigma}(D)\big)$.
		To this end, by Remark \ref{rem:NZD_properties}(1) we know that $\Supp N_\sigma(D) = B$, so using that $H \cdot B = 3$, $B^2 = -4$ and $ P_\sigma(D) \cdot N_\sigma(D) = 0 $, we find
		\[ P_\sigma(D) = H + \frac{3}{4}B \quad \text{ and } \quad N_\sigma(D) = \frac{1}{4}B , \]
		and hence Remark \ref{rem:abs_alt_description} yields
		$ \abs(D) = \abs \big( P_\sigma(D) \big) = B $.
		
		In conclusion, 
		\[ \dbs(D) = \sbs(D) = \abs(D) = B \cong B_0 \]
		is not uniruled, whereas $X$ is uniruled.
	\end{exa}
	
	In the remainder of this section we deal with Corollaries \ref{cor:BBP13_CorA_g} and \ref{cor:loci_uniruledness_abundant}.
	
	\begin{proof}[Proof of Corollary \ref{cor:BBP13_CorA_g}]~
		
		\medskip
		
		\noindent (i) If $V$ is an irreducible component of $\dbs(D)$, then by \cite[Lemma 1.14]{ELMNP06} we deduce that $V$ is an irreducible component of $\abs(D+A)$ for some ample $\R$-divisor $A$ on $X$. Therefore, since $D+A$ is big, we obtain (i) by arguing as in the proof of Corollary \ref{cor:loci_comparison_arb_nef_antican}(i) and by applying Theorem \ref{thm:loci_uniruledness_can}(ii) instead of Theorem \ref{thm:loci_comparison_can}(iii).
		
		\medskip
		
		\noindent (ii) We argue as in the proof of Corollary \ref{cor:loci_comparison_arb_nef_antican}(ii), except that we now apply Theorem \ref{thm:loci_uniruledness_can}(ii) instead of Theorem \ref{thm:loci_comparison_can}(iii).
		
		\medskip
		
		\noindent (iii) We conclude by (i) and Corollary \ref{cor:loci_comparison_arb_nef_antican}(iii).
	\end{proof}
	
	\begin{proof}[Proof of Corollary \ref{cor:loci_uniruledness_abundant}]~
		
		\medskip
		
		\noindent (i) We conclude by Corollary \ref{cor:loci_comparison_abundant}(i) and Theorem \ref{thm:loci_uniruledness_can}(i).
		
		\medskip
		
		\noindent (ii) We conclude by Corollary \ref{cor:loci_comparison_abundant}(ii) and Corollary \ref{cor:BBP13_CorA_g}(i).
	\end{proof}
	
	The last two examples in this section indicate that in the setting of Corollary \ref{cor:BBP13_CorA_g}(i) the irreducible components of $\sbs(D)$ need not be uniruled in general; similarly for $\abs(D)$, unless $X$ itself is uniruled.
	
	\begin{exa}~
		\label{exa:uniruledness_fails_II}
		
		\medskip
		
		\noindent (i) \cite[Example 1.1]{Sho00}:
		Let $ X \coloneqq \mathbb{P}(\mathcal{E}) \to C$ be a $\mathbb{P}^1$-bundle over an elliptic curve $C$, where $\mathcal{E}$ is a rank-two vector bundle over $C$ which is defined by an extension
		\[ 0 \to \mathcal{O}_C \to \mathcal{E} \to \mathcal L \to 0 , \] 
		where $ \mathcal L$ a non-torsion line bundle on $C$ of degree zero. Then $-K_X$ is nef. Moreover, we have $\kappa(X,-K_X)=0$ and the unique member in the linear system $|{-}K_X|$ is $D \coloneqq C_1+C_2$, where $C_1$ and $C_2$ are disjoint elliptic curves. It follows that $ \dbs(C_1) = \dbs(C_2) = \emptyset $,
		$\sbs(C_1)=C_1$ and $\sbs(C_2)=C_2$ are not uniruled, whereas $\abs(C_1)=\abs(C_2)=X$ is uniruled.
		
		\medskip
		
		\noindent (ii) Let $E$ be an elliptic curve and pick a non-torsion divisor $Q$ on $E$ of degree $0$. Consider the abelian surface $X \coloneqq E \times E$, denote by $\operatorname{pr}_1 \colon X \to E$ the canonical projection to the first factor and set $D \coloneqq \operatorname{pr}_1^* Q$. Since $\kappa(X,D) = \kappa(E,Q) = - \infty$, we have $ \sbs(D) = \abs(D) = X $, which is not uniruled.
	\end{exa}

	\section{On the uniruledness of the diminished base locus: the lc case}
	\label{section:uniruledness_dbs_lc}
	
	Theorem \ref{thm:loci_uniruledness_can}(i) and the relevant discussion in Section \ref{section:uniruledness_loci} lead naturally to the following question.
	
	\begin{question}
		\label{question:uniruledness_dbs}
		Let $(X,B+M)$ be an lc g-pair such that $K_X+B+M$ is pseudo-effective. Is every irreducible component of $\dbs(K_X+B+M)$ uniruled?
	\end{question}
	
	We can give the following affirmative answer to Question \ref{question:uniruledness_dbs}.
	
	\begin{thm}
		\label{thm:uniruledness_dbs_can_lc}
		If $(X,B+M)$ is an lc g-pair such that $K_X+B+M$ is pseudo-effective and if $(X,0)$ is klt, then every irreducible component of $\dbs(K_X+B+M)$ is uniruled.
	\end{thm}
	
	\begin{proof}
		Let $V$ be an irreducible component of $\dbs(K_X+B+M)$. Arguing as in the proof of Theorem \ref{thm:loci_comparison_can}(i), we infer that there exists an ample $\R$-divisor $A$ on $X$ such that $V$ is an irreducible component of $\dbs\big( K_X+(B+A)+M\big)$. Since $(X,0)$ is a klt pair, by \cite[Lemma 3.4]{HanLi22} there exists a boundary $\Delta \sim_\R B+A+M$ on $X$ such that $(X,\Delta)$ is a klt pair, so $V$ is an irreducible component of $\dbs(K_X+\Delta)$ by Remark \ref{rem:abs+dbs_properties}(1). Therefore, by applying Theorem \ref{thm:loci_uniruledness_can}(i) to the pair $(X,\Delta)$, we conclude that $V$ is uniruled.
	\end{proof}
	
	As an application of Theorem \ref{thm:uniruledness_dbs_can_lc}, we prove that the irreducible components of the diminished base locus of the $\R$-divisor $D$ from Corollary \ref{cor:loci_comparison_arb_lc_partial} are uniruled.
	
	\begin{cor}
		\label{cor:BBP13_CorA_lc_dbs}
		Let $(X,B+M)$ be an lc g-pair such that $K_X+B+M \equiv 0$. Assume that $(X,0)$ is klt, and let $D \in \Div_{\R}(X)$ with $|D|_\R \neq \emptyset$. If there exists $G \in |D|_\R$ such that 
		\[ \tau \coloneqq \sup \{ t \in [0,1] \mid K_X+ (B+tG) + M \emph{ is lc} \} > 0 , \]
		then every irreducible component of $\dbs(D)$ is uniruled.
	\end{cor}
	
	\begin{proof}
		The statement follows immediately from Theorem \ref{thm:uniruledness_dbs_can_lc}, applied to the lc g-pair $\big( X, (B+\tau G) + M \big)$, taking Remark \ref{rem:abs+dbs_properties}(1)(2) into account.
	\end{proof}
	
	The next example shows that the assumption in Corollary \ref{cor:BBP13_CorA_lc_dbs} that the lc threshold $\tau \coloneqq \operatorname{lct} (X, B+M;G)$ is positive for some $ G \in |D|_\R $ cannot be removed, demonstrating thus that our previous result is sharp.
	
	\begin{exa}
		Let $\mu \colon X \to S$, $E$ and $H$ be as in Example \ref{exa:blowup_vertex_of_cone}. Set $B \coloneqq E$ and $M \coloneqq \mu^*H$. Then $(X,B+M)$ is a log smooth lc polarized $\Q$-pair (which is not klt) such that $K_X+B+M \sim_\R 0 $. According to \cite[Example 6.4]{Tak08}, $-K_X$ is big but not nef, and 
		\[ \dbs({-}K_X) = \sbs({-}K_X) = \abs({-}K_X) = E \]
		is not uniruled, whereas $X$ is uniruled. Observe also that the g-pair $\big( X, (B+tG) + M \big)$ is not lc for any $G \in |{-}K_X|_\Q$, since $\sbs({-}K_X) = E$.
	\end{exa}

	\bibliographystyle{amsalpha}
	\bibliography{BibliographyForPapers}
	
\end{document}